\documentclass[12pt]{amsart}
\usepackage{amsmath, amssymb, amsfonts, amsthm, latexsym}
\usepackage{graphicx}
\usepackage{tikz}

\newtheorem{Theorem}{Theorem}[section]
\newtheorem{Lemma}[Theorem]{Lemma}
\newtheorem{Corollary}[Theorem]{Corollary}
\newtheorem{Proposition}[Theorem]{Proposition}

\theoremstyle{definition}
\newtheorem{Remark}[Theorem]{Remark}
\newtheorem{Example}[Theorem]{Example}
\newtheorem{Definition}[Theorem]{Definition}

\def\Ass{\operatorname{Ass}}
\def\reg{\operatorname{reg}}
\def\depth{\operatorname{depth}} 
\def\astab{\operatorname{astab}} 
\def\supp{\operatorname{supp}}

\def\NN{{\mathbb N}}
 
\def\a{{\mathbf a}}
\def\b{{\mathbf b}}

\def\e{{\mathbf e}}
\def\1{{\mathbf 1}}
\def\mm{{\mathfrak m}}
\def\nn{{\mathfrak n}}
\def\pp{{\mathfrak p}}
\def\E{{\mathcal E}}
\def\F{{\mathcal F}}

\def\A{{\mathcal A}}
\def\B{{\mathcal B}}
\def\D{{\Delta}}
\def\G{{\Gamma}}

\begin{document}

\title{Associated primes of powers of edge ideals\\
and ear decompositions of graphs}

\author{Ha Minh Lam}
\address{Institute of Mathematics, Vietnam Academy of Science and Technology, 18 Hoang Quoc Viet, 10307 Hanoi, Vietnam}
\email{hmlam@math.ac.vn}

\author{Ngo Viet Trung}
\address{Institute of Mathematics, Vietnam Academy of Science and Technology,18 Hoang Quoc Viet, 10307 Hanoi, Vietnam}
\email{nvtrung@math.ac.vn}


\begin{abstract}
In this paper, we give a complete description of the associated primes of every power of the edge ideal in terms of generalized ear decompositions of the graph. This result establishes a surprising relationship between two seemingly unrelated notions of Commutative Algebra and Combinatorics. It covers all previous major results in this topic and has several interesting consequences.
\end{abstract}

\maketitle


\section*{Introduction}

This work is motivated by the asymptotic properties of powers of a graded ideal $Q$ in a graded algebra $S$ over a field. It is known that for $t$ sufficiently large, the depth of $S/Q^t$ is a constant \cite{Br2} and the Castelnuovo-Mumford regularity of $Q^t$ is a linear function \cite{CHT}, \cite{Ko}. 
These results have led to recent works on the behavior of the whole functions $\depth S/Q^t$ \cite{HH} and $\reg Q^t$ \cite{EH}, \cite{EU}. Inevitably, one has to address the problem of estimating $\depth S/Q^t$ and $\reg Q^t$ for initial values of $t$. This problem is hard because there is no general approach to study a particular power $Q^t$. 
In order to understand the general case one has to study ideals with additional structures. \par

Let $R = k[x_1,...,x_n]$ be a polynomial ring over a field $k$.
Given a hypergraph $\G$ on the vertex set $V = \{1,...,n\}$, 
one calls the ideal $I$ generated by the monomials $\prod_{i \in F}x_i$, where $F$ is an edge of $\G$, the {\em edge ideal} of $\G$.
This notion has provided a fertile ground to study powers of ideals because of its link to combinatorics (see e.g.  \cite{HH1}, \cite{HHT}, \cite{MV}, \cite{SVV}). However, 
there were only a few works on the behavior of the functions $\depth R/I^t$ and $\reg I^t$ for special classes of edge ideals \cite{Ba}, \cite{HS}, \cite{HQ}, \cite{Mo}, \cite{TNT}. \par

It is known that the depth and the Castelnuovo-Mumford regularity can be defined in terms of the local cohomology modules. By a result of Takayama \cite{Ta}, the local cohomology modules of $R/I^t$ can be computed by means of the reduced homology of certain simplicial complexes. The facets of these simplicial complexes correspond to associated primes of $I^t$ \cite{TT2}. Therefore, it is necessary to know the associated primes of $I^t$ in order to compute $\depth R/I^t$ and $\reg I^t$. 

Associated primes $I^t$  are of interest in their own right. For instance, Simis, Vasconcelos and Villarreal \cite{SVV} or Francisco, Ha and Van Tuyl \cite{FHV} showed that they can be used to detect odd cycles in a graph or to give an algebraic characterization of perfect graphs without using the Strong Perfect Graph Theorem. \par

For an ideal $I$ in a noetherian ring, let $\Ass(I^t)$ denote the associated primes of $I^t$. 
Brodmann \cite{Br} showed that there is an integer $t_0$ such that $\Ass(I^t) = \Ass(I^{t+1})$ for all $t \ge t_0$. 
Let $\Ass^\infty(I)$ denote the stable set $\Ass(I^{t_0})$, and let $\astab(I)$ be the least number $t_0$ with this property. 
In general, it is very hard to compute $\Ass^\infty(I)$ and $\astab(I)$, even for a monomial ideal $I$ \cite{Ho}. \par

Let $I$ be the edge ideal of a graph $\G$. Chen, Morey and Sung \cite{CMS} gave an algorithmic construction of $\Ass^\infty(I)$ and an upper bound for $\astab(I)$. 
The problem here is to describe the associated primes of initial powers of $I$ in terms of $\G$. The associated primes of $I^t$ are of the form $P_F = (x_i|\ i \in F)$, where $F$ is a (vertex) cover of $\G$. 
In particular, the minimal associated primes of $I^t$ correspond to the minimal covers of $\G$. 
Therefore, it remains to describe the covers $F$ of $\G$ for which $P_F$ is a non-minimal associated prime of $I^t$. 
Such a prime ideal is usually called an {\em embedded prime} of $I^t$. \par

For $t = 2$, the above problem was solved by Terai and Trung \cite{TT2}. They showed that $P_F$ is an embedded prime of $I^2$ if and only if $F$ is minimal among the covers containing the closed neighborhood of a triangle. A somewhat weaker result was obtained independently by Herzog and Hibi in \cite{HH2}, who proved that the maximal homogeneous ideal $\mm$  is an associated prime of $I^2$ if and only if $\G$ has a dominating triangle. Recently,  Hien, Lam and Trung \cite{HLT} and Hien and Lam \cite{HL} succeeded in giving a complete classification of the embedded primes of $I^3$ and $I^4$ in terms of $\G$. These results are so complicated that a classification for $t \ge 5$ appears to be impossible. 
Associated primes of powers of the more general class of squarefree monomial ideals have been described by Ha and Morey \cite{HM} and Francisco, Ha and Van Tuyl \cite{FHV} from a different point of view. However, their works could not be used to describe the associated primes of $I^t$ in terms of $\G$. \par
\par 

In this paper we will give a complete description of the embedded primes of $I^t$ for every $t \ge 2$ by using 
the following combinatorial notion. 
We call a sequence of walks without repetition of the vertices other than the endpoints a {\em generalized ear decomposition} of the graph if the first walk is closed, the endpoints of each subsequent walk are the only vertices of that walk belonging to earlier walks, and the walks pass through all vertices of the graph. 
This notion is a modified version of ear decomposition, a well-known notion in graph theory (see e.g.~\cite{Lo, Ro}). 
Note that a walk without repetition of the vertices other than the endpoints is a path or a cycle or a closed walk of length 2 (a repetitive edge), which look like ears when pasting them together. 
Generalized ear decompositions always exist in a connected graph. In particular, every cycle can be used as the first walk of a generalized ear decomposition.

We call a graph {\em strongly non-bipartite} if every connected component is non-bipartite or, equivalently, contains an odd cycle. If the first walk of a generalized ear decomposition is an odd cycle, we call it an {\em odd-beginning generalized ear decomposition}.   
For a strongly non-bipartite graph $\G$,  we introduce the invariant 
$$\mu^*(\G) : = (\varphi^*(\G)+n-c)/2,$$
where $\varphi^*(\G)$ is the minimal number of walks of even length in families of odd-beginning generalized ear decomposition of the connected components, and $c$ is the number of the connected components of $\G$. 
The notation $\mu^*(\G)$ is inspired by the invariant $\mu(\G)$, which was introduced in coding theory by Sole and  Zaslavsky \cite{SZ} and studied in hypergraph theory by Frank \cite{Fr}. \par

Our main result can be stated as follows.  \medskip

\noindent {\bf Theorem \ref{associated}}. 
Let $F$ be a cover of $\G$. Then $P_F$ is an associated prime of $I^t$ if and only if $F$ is either a minimal cover or $F$ is minimal among the covers containing the closed neighborhood $N[U]$ of a subset $U \subseteq V$ such that the induced subgraph $\G_U$ is strongly non-bipartite with $\mu^*(\G_U) < t$. 
\medskip

By Theorem \ref{associated}, every embedded associated prime of $I^t$ originates from a strongly non-bipartite induced subgraph. To find these subgraphs one only need to look at the odd cycles of $\G$.
This provides a simple way  to {\em determine the embedded primes of every power $I^t$ at the same time}.
\medskip

\noindent{\bf Example}. 
Let $\G$ be the graph in Figure 1. Then $\G$ has only two odd cycles on the sets $U_1 = \{1,2,3\}$ and $U_2 = \{2,3,4,5,6\}$. 

\begin{figure}[ht!]
\begin{tikzpicture}[scale=0.6]

\draw [thick] (0.5,1) coordinate (a) -- (2,2) coordinate (b) ;
\draw [thick] (2,2) coordinate (b) -- (2,0) coordinate (c) ;
\draw [thick] (2,0) coordinate (c) -- (0.5,1) coordinate (a) ; 
\draw [thick] (2,2) coordinate (b) -- (4,2) coordinate (d);
\draw [thick] (2,0) coordinate (c) -- (4,0) coordinate (e);
\draw [thick] (4,0) coordinate (e) -- (5.5,1) coordinate (f);
\draw [thick] (4,2) coordinate (d) -- (5.5,1) coordinate (f);
    
\fill (a) circle (3pt);
    \fill (b) circle (3pt);
  \fill (c) circle (3pt);
  \fill (d) circle (3pt);
  \fill (e) circle (3pt);
  \fill (f) circle (3pt);
  
\draw (0.5,1) node[left] {$1$};
\draw (2,2) node[left, above] {$2$};
\draw (2,0) node[left, below] {$3$};
\draw (4,2) node[right,above] {$4$};
\draw (4,0) node[right,below] {$5$};
\draw (5.5,1) node[right] {$6$};

\end{tikzpicture}
\caption{}
\end{figure}

Any subset $U \subseteq V$ such that $\G_U$ is strongly non-bipartite must contain $U_1$ or $U_2$. 
It is clear that $N[U]$ is either $F_1 = \{1,2,3,4,5\}$ or $F_2 = \{1,2,3,4,5,6\}$. 
Since $F_1$ and $F_2$ are covers of $\G$, which are minimal among the covers containing themselves,
the embedded primes of $I^t$ can be only $P_{F_1} = (x_1,...,x_5)$ or $P_{F_2} = (x_1,...,x_6)$. 
We have $\mu^*(\G_{U_1}) = 1$, $\mu^*(\G_{U_2}) = 2$, and $\mu^*(\G_U) \ge 2$ for all other such $U$. 
Since $F_1 = N[U_1]$ and $F_2 = N[U_2]$, we conclude that $I^2$ has only an embedded prime $P_{F_1}$ and $I^t$, $t \ge 3$, has two embedded primes $P_{F_1}$ and $P_{F_2}$. \smallskip

It is easy to see that the triangle is the only strongly non-bipartite graph whose $\mu^*$-invariant equals 1. 
Therefore, the aforementioned results of Herzog and Hibi \cite{HH2} and of Terai and Trung \cite{TT2} on embedded 
primes of $I^2$ are special cases of Theorem \ref{associated}. We also give a recursive description of strongly non-bipartite graphs with a given $\mu^*$-invariant, which can be used to classify  embedded primes of $I^t$ for $t \ge 3$. \par

In a pioneering paper on edge ideals \cite{SVV}, Simis, Vasconcelos, and Villarreal showed that $I^t$ has no embedded primes for all $t$ if and only if $\G$ is a bipartite graph. This result is a straightforward consequence of Theorem \ref{associated} because the existence of embedded primes of $I^t$ for some $t \ge 1$ depends on the existence of an odd cycle. \par

From Theorem \ref{associated} it immediately follows that if $P_F \in \Ass(I^t)$, then $P_F \in \Ass(I^{t+1})$. That means $\Ass(I^t) \subseteq \Ass(I^{t+1})$ for all $t \ge 1$. This property was a major result on $\Ass(I^t)$ obtained by Martinez-Bernal, Morey and Villarreal \cite{MMV}. 
\par

Another immediate consequence of Theorem \ref{associated} is that $\Ass^\infty(I)$ is the set of the prime ideals $P_F$ for which $F$ is either a minimal cover or minimal among the covers containing the closed neighborhood of a subset $U \subseteq V$ with the property that $\G_U$ is strongly non-bipartite. This explicit description of  $\Ass^\infty(I)$ was found recently by Hien, Lam and Trung \cite{HLT}.
Furthermore, we can give a precise formula for $\astab(I)$ and even for the index of stability of every associated prime of $\Ass^\infty(I)$, thereby answering a question raised by Sharp \cite{Sh} for edge ideals. \par

The proof of Theorem \ref{associated} occupies most of our paper. 
The basic idea comes from \cite{HLT}, where embedded primes of $I^t$ are characterized by the existence of (vertex) weighted graphs with special matching properties. Our new idea is that these weighted graphs can be replaced by  matching-critical weighted graphs, which are modifications of matching-critical graphs in graph theory. By results of Gallai \cite{Ga} and Lovazs \cite{Lo}, matching-critical graphs can be characterized by means of ear decompositions.
Using their results we can describe the existence of matching-critical weighted graphs in terms of generalized ear decompositions of the base graph. Matching-critical graphs have been also used to describe irreducible decompositions of powers of $I^t$ \cite{DTTY}. \par	

We believe that generalized ear decompositions are an useful new tool for the study of edge ideals.
The combinatorial description of the associated primes of $I^t$ is only the first step in the computation of $\depth R/I^t$ and $\reg I^t$. The next step is to describe the simplicial complexes appearing in Takayama's formula for the local cohomology modules of $R/I^t$ by using the results of this paper. \par

The paper is organized as follows. In Section 1 we describe the monomials of the socle of the ring $R/I^t$ in terms of weighted graphs. The existence of such a monomial is a criterion for $\mm \in \Ass(I^t)$. This description leads to the study of matching-critical weighted graphs and their relationship with generalized ear decompositions in Section 2. In Section 3 we use generalized ear decompositions to characterize the base graphs of matching-critical weighted graphs with a given matching number. The characterization of the embedded primes of $I^t$ and its applications are given in Section 4. We conclude the paper with Section 5, where we recursively describe the minimal strongly non-bipartite graphs with a given $\mu^*$-invariant and show how to use this description to classify the associated primes of $I^t$. \par

Throughout this paper, $\G$ denotes a simple graph on the vertex set $V = \{1,...,n\}$ and $I$ the edge ideal of $\G$ in the polynomial ring $R = k[x_1,...,x_n]$. For unexplained terminology in commutative algebra and in graph theory we refer to \cite{E} and \cite{We}, respectively. \par
\medskip

\noindent {\bf Acknowledgement}.
This work is supported by Vietnam National Foundation for Science and Technology Development under grant number 101.04-2017.19.  
The authors are thankful to Amin Seyed Fakhari, Dariush Kiani and Sara Saeedi Madani for pointing out that 
the weighted graph considered Proposition \ref{socle 2} and Theorem \ref{max 2} must have no isolated vertices. 
They are also thankful to the anonymous referee for many suggestions which have improved the presentation of this paper.


\section{Socle of powers of edge ideals}

Let $\mm$ denote the maximal homogeneous ideal of $R$.
The ideal $(I^t :\mm)/I^t$ is called the {\em socle} of the quotient ring $R/I^t$. 
It is well known that $\mm \in \Ass(I^t)$ if and only if $(I^t :\mm)/I^t \neq 0$. 
In this section, we will express this condition in combinatorial terms and deduce necessary and sufficient conditions for
$\mm \in \Ass(I^t)$.
\par

Let $\NN$ denote the set of non-negative integers. 
For a vector $\a = (a_1,...,a_n) \in \NN^n$ we denote by $x^\a$ the monomial $x_1^{a_1} \cdots x_n^{a_n}$.
Since $I^t$ and $I^t :\mm$ are monomial ideals, $(I^t :\mm)/I^t \neq 0$ if and only if there exists a monomial $x^\a \in (I^t :\mm) \setminus I^t$. This condition means $x^\a \not \in I^t$ and $x^{\a+\e_i} \in I^t$ for all $i \in V$, where $\e_i$ denotes the $i$-th unit vector in $\NN^n$.  
These relations can be translated in combinatorial terms by using the following notions.
\par 

Let $\G_\a$ denote the weighted graph on $V$ in which every vertex $i$ is assigned the weight $a_i$ and 
whose edges are edges of $\G$ having vertices of positive weights.

\begin{Definition}
A {\em matching} of $\G_\a$ is a family of not necessarily different edges such that 
$a_i$ is greater than or equal to the number of times the vertex $i$ appears in these edges for all $i \in V$. 
The largest possible number of edges in a matching of $\G_\a$ is called the {\em matching number} of $\G_\a$ and denoted by $\nu(\G_\a)$. 
\end{Definition}

If $\a = (1,...,1)$,  the edges of a matching of $\G_{(1,...,1)}$ are disjoint because every vertex appears at most once in these edges. Therefore, a matching of $\G_{(1,...,1)}$ is a usual matching of $\G$ and we may identify $\G_{(1,...,1)}$ with $\G$. 

\begin{Lemma}   \label{member}
 $x^\a \in I^t$ if and only if $\nu(\G_\a) \ge t$.
\end{Lemma}

\begin{proof}
We have $x^\a \in I^t$ if and only if there is a family of $t$ edges $\{i_1,j_1\},...,\{i_t,j_t\}$  of $\G$ such that
$x^\a$ is divisible by $\prod_{r=1}^t x_{i_r}x_{j_r}$. 
Comparing the exponents of these two monomials, we can express this condition as
$$\a \ge \sum_{r=1}^t (\e_{i_r} + \e_{j_r})$$
componentwise. 
This relation is satisfied if and only if $a_i$ is greater than or equal to the number of times the vertex $i$ appears in the edges $\{i_1,j_1\},...,\{i_t,j_t\}$ for all $i \in V$. Therefore, $x^\a \in I^t$ if and only if $\G_\a$ has a matching of $t$ edges.
\end{proof}

When dealing with matchings of the weighted graph $\G_\a$ we can pass to the following simple graph. \par

\begin{Definition} (see e.g.~\cite{MMV})
The {\em parallelization} $\G^\a$ of $\G$ is the graph on $a_1\cdots a_n$ vertices $1_1,...,1_{a_1}$,...,$n_1,...,n_{a_n}$, where $\{i_r,j_s\}$ is an edge of $\G^\a$ if $\{i,j\}$ is an edge of $\G$.  
\end{Definition}

Simply speaking, $\G^\a$ is obtained from $\G$ by replacing each vertex $i$ of $\G$ by $a_i$ new vertices $i_1,...,i_{a_i}$ and each edge $\{i,j\}$ of $\G$ by $a_ia_j$ new edges $\{i_r,j_s\}$, $r = 1,...,a_i, s = 1,...,a_j$.

\begin{Example} \label{simple}
Let $\G$ be the following graph on 4 vertices and $\a = (1,1,2,1)$. Then $\G_\a$ and $\G^\a$ can be illustrated as in Figure I.
\end{Example}

\begin{figure}[ht!]

\begin{tikzpicture}[scale=0.6] 

\draw [thick] (0,0) coordinate (a) -- (0,2) coordinate (b) ;
\draw [thick] (0,2) coordinate (b) -- (1.5,1) coordinate (c) ;
\draw [thick] (1.5,1) coordinate (c) -- (0,0) coordinate (a) ; 
\draw [thick] (1.5,1) coordinate (c) -- (3.5,1) coordinate (d);

\draw (2.25,-1) node{$\Gamma$};

\draw (3.5,1) node[right] {$4$};
\draw (-0.2,0) node[left, below] {$1$};
\draw (-0.2,2) node[left, above] {$2$};
\draw (1.5,1) node[above] {$3$};    
\fill (a) circle (3pt);
  \fill (b) circle (3pt);
  \fill (c) circle (3pt);
  \fill (d) circle (3pt);
 
\draw [thick] (8,0) coordinate (a) -- (8,2) coordinate (b) ;
\draw [thick] (8,2) coordinate (b) -- (9.5,1) coordinate (c) ;
\draw [thick] (9.5,1) coordinate (c) -- (8,0) coordinate (a) ; 
\draw [thick] (9.5,1) coordinate (c) -- (11.5,1) coordinate (d);

\draw (10.25,-1) node{$\G_\a$};

\draw (11.5,1) node[right] {$4$};
\draw (7.8,0) node[left, below] {$1$};
\draw (7.8,2) node[left, above] {$2$};
\draw (10.2,1) node[right,above] {$3(2\times)$};    
\fill (a) circle (3pt);
  \fill (b) circle (3pt);
  \fill (c) circle (3pt);
  \fill (d) circle (3pt);
  
\draw [thick] (16,0) coordinate (a) -- (16,2) coordinate (b) ;
\draw [thick] (16,2) coordinate (b) -- (17.5,1) coordinate (c) ;
\draw [thick] (17.5,1) coordinate (c) -- (16,0) coordinate (a) ; 
\draw [thick] (17.5,1) coordinate (c) -- (19.5,1) coordinate (d);
\draw [thick] (16,2) coordinate (b) -- (17.5,3) coordinate (e) ;
\draw [thick] (16,0) coordinate (a) -- (17.5,3) coordinate (e) ;
\draw [thick] (17.5,3) coordinate (e) -- (19.5,1) coordinate (d);

\draw (18.25,-1) node{$\G^\a$};

\draw (19.5,1) node[right] {$4$};
\draw (15.8,0) node[left, below] {$1$};
\draw (15.8,2) node[left, above] {$2$};
\draw (17.5,1) node[right,above] {$3_1$}; 
\draw (17.5,3) node[right,above] {$3_2$};   
\fill (a) circle (3pt);
  \fill (b) circle (3pt);
  \fill (c) circle (3pt);
  \fill (d) circle (3pt);
  \fill (e) circle (3pt);
\end{tikzpicture}
\caption{}
\end{figure}

We call the map which sends every vertex $i_r$ of $\G^\a$ to the vertex $i$ of $\G$ the {\em projection} of $\G^\a$ to $\G$. Note that if $\{i_r,j_s\}$ is an edge of $\G^\a$ , then its projection $\{i,j\}$ is an edge of $\G$.
Using this map one can deduce properties of the weighted graph $\G_\a$ from those of the simple graph $\G^\a$.

\begin{Lemma}   \label{matching}
$\nu(\G_\a) = \nu(\G^\a)$.
\end{Lemma}

\begin{proof}
We only need to show that there is a correspondence between matchings of $\G_\a$ and $\G^\a$ with the same cardinality. 
For every matching $M$ of $\G_\a$, we know that $a_i$ is greater than or equal to the number of times the vertex $i$ appears in $M$ for all $i \in V$. Using this fact, we can easily find a matching $N$ of $\G^\a$ with the same cardinality such that $M$ is the projection of $N$. Conversely, for every matching $N$ of $\G^\a$, the projection of $N$ is a family of edges in $\G$. Since the edges of $N$ are disjoint, $a_i$ is greater than or equal to the number of times the vertex $i$ appears in the projection of $N$. Therefore, the projection of $N$ is a matching of $\G_\a$ with the same cardinality. 
\end{proof}

Let $\supp(\a) := \{i \in V|\ a_i \neq 0\}$ be the support of $\a$. 
We may consider $\supp(\a)$ as the vertex set of $\G_\a$. 
Let $\G_{\supp(\a)}$ denote the induced graph of $\G$ on $\supp(\a)$. 
We say that $\G_\a$ has no isolated vertex if $\G_{\supp(\a)}$ has no isolated vertex.
For every $i \in \supp(\a)$, we have a weighted graph $\G_{\a-\e_i}$ because $\a-\e_i \in \NN^n$. \par

For a subset $U$ of $V$ we denote by $N[U]$ the {\em closed neighborhood} of $U$ in $\G$, i.e. the union of $U$ and the set of the vertices adjacent to vertices of $U$. If $V = N[U]$, one calls $U$ a {\em dominating set} of $\G$.

\begin{Lemma} \label{socle 1} 
Assume that $x^\a \in (I^t:\mm) \setminus I^t$, $t \ge 2$. Then \par
{\rm (i) } $\supp(\a)$ is a dominating set of $\G$,\par
{\rm (ii)}  $\G_\a$ has no isolated vertex,\par
{\rm (iii)} $\nu(\G_\a) = t-1$,\par
{\rm (iv)} either $\nu(\G_{\a-\e_i}) = t-1$ for all $i \in \supp(\a)$ or $x^{\a-\e_i} \in (I^{t-1}:\mm) \setminus I^{t-1}$ for some $i \in \supp(\a)$. 
\end{Lemma}

\begin{proof}   
The assumption implies $x^\a \not \in I^t$ and $x^{\a+\e_i} \in I^t$ for all $i \in V$.
By Lemma \ref{member}, these conditions can be expressed as $\nu(\G_\a) < t$ and $\nu(\G_{\a+\e_i}) \ge  t$ for all $i \in V$.  \par
(i)  For all $i \in V$, we have $\nu(\G_{\a+\e_i}) > \nu(\G_\a)$. From this it follows that there exists a matching of $\G_{\a+\e_i}$ which has an edge containing $i$. Hence, $i$ is adjacent to a vertex of $\supp(\a)$.
This shows that $V = N[\supp(\a)]$. \par
(ii) Assume that $\G_\a$ has an isolated vertex $i$. Since there is no edge of $\G_\a$ containing $i$, we get
$\nu(\G_{\a+\e_i}) = \nu(\G_\a)$, which is a contradiction. \par
(iii) By Lemma \ref{matching}, $\nu(\G_\a) = \nu(\G^\a)$ and $\nu(\G_{\a+\e_i}) = \nu(\G^{\a+\e_i})$ for all $i \in V$. Since $\G^\a$ can be obtained from $\G^{\a+\e_j}$  by deleting a vertex, $\nu(\G^\a) \ge \nu(\G^{\a+\e_i}) - 1$. Therefore, $\nu(\G_\a) \ge  \nu(\G_{\a+\e_i}) - 1 \ge t-1$. Since $\nu(\G_\a) < t$, this implies $\nu(\G_\a) = t-1$. \par
(iv) Similarly as above, 
we have $\nu(\G_{\a-\e_i+\e_j}) \ge \nu(\G_{\a+\e_j}) - 1 \ge t-1$ for all $i \in \supp(\a)$ and $j \in V$. 
By Lemma \ref{member},  this implies $x^{\a-\e_i+\e_j} \in I^{t-1}$. Hence, $x^{\a-\e_i} \in I^{t-1}: \mm$.
Since every matching of $\G_{\a-\e_i}$ is also a matching of $\G_\a$, we have $\nu(\G_{\a-\e_i}) \le \nu(\G_\a) = t-1$.  
If $\nu(\G_{\a-\e_i}) \neq t-1$ for some $i \in \supp(\a)$, we have $\nu(\G_{\a-\e_i}) < t-1$. Hence,
$x^{\a-\e_i} \not\in I^{t-1}$ by Lemma \ref{member}.  In this case, $x^{\a-\e_i} \in (I^{t-1}:\mm) \setminus I^{t-1}$.   
\end{proof}

Using the non-vanishing of the socle of $R/I^t$ we obtain the following necessary and sufficient conditions for $\mm \in \Ass(I^t)$, which are based on the existence of weighted graph $\G_\a$ with special properties.

\begin{Proposition} \label{max 1}  
Assume that $\mm \in \Ass(I^t) \setminus \Ass(I^{t-1})$, $t \ge 2$. Then there exists a weighted graph $\G_\a$ without isolated vertices such that $\supp(\a)$ is a dominating set of $\G$ and $\nu(\G_{\a-\e_i}) = \nu(\G_\a) = t-1$  for all $i \in \supp(\a)$. 
\end{Proposition}

\begin{proof}   
As observed at the beginning of this section, the condition $\mm \in \Ass(I^t)$ implies the existence of a monomial $x^\a \in (I^t:\mm) \setminus I^t$. Similarly, the condition $\mm \not\in \Ass(I^{t-1})$ implies that there does not exist $i \in \supp(\a)$ such that $x^{\a-\e_i} \in (I^{t-1}:\mm) \setminus I^{t-1}$. Therefore, the conclusion follows from 
Lemma \ref{socle 1}.
\end{proof}

\begin{Proposition} \label{socle 2} 
Assume that there exists a weighted graph $\G_\a$ without isolated vertices such that $\supp(\a)$ is a dominating set of $\G$ and $\nu(\G_{\a-\e_i}) = \nu(\G_\a) = t-1$  for all $i \in \supp(\a)$.
Then $\mm \in \Ass(I^t)$. 
\end{Proposition}

\begin{proof}   
As observed at the beginning of this section, we only need to show that $x^\a \in (I^t:\mm) \setminus I^t$.
By Lemma \ref{member}, $\nu(\G_\a) = t-1$ implies $x^\a \not\in I^t$.
It remains to show that $x^\a \in I^t:\mm$ or, equivalently, $x^{\a+\e_i} \in  I^t$ for all $i \in V$. 
Since $\supp(\a)$ is a dominating set of $\G_\a$ and $\G_\a$ has no isolated vertices, 
we can find a vertex $j \in \supp(\a)$ adjacent to $i$. 
Since $\nu(\G_{\a-\e_j}) = t-1$, $\G_{\a-\e_j}$ has a matching of $t-1$ edges.
Adding the edge $\{i,j\}$ to this matching we obtain a matching of $\G_{\a+\e_i}$. 
Thus, $\nu(\G_{\a+\e_i}) \ge t$. By Lemma \ref{member}, this implies $x^{\a+\e_i} \in  I^t$. 
\end{proof}

\begin{Remark}
{\rm By \cite[Proposition 2.1]{HH2}, 
the condition $x^\a \in (I^t:\mm) \setminus I^t$ implies $a_i \le t-1$ for all $i \in V$.
 (Using \cite[Theorem 1]{Ta} one can also prove this fact for the edge ideal of a hypergraph.) 
If $t = 2$, we have $a_i \le 1$ for all $i \in V$. 
Hence, we may identify $\G_\a$ with the induced graph $\G_{\supp(\a)}$. 
This explains why the case $t=2$ can be studied without involving weighted graphs \cite{HH2}, \cite{TT2}. 
If $t \ge 3$,  we really need to investigate weighted graphs \cite{HL}, \cite{HLT}. The new idea here is that we take into account the weighted graphs $\G_{\a-\e_i}$. This idea allows us to relate the condition $\mm \in \Ass(I^t)$ to well known notions in graph theory as we shall see in the next section.}
\end{Remark}


\section{Matching-critical weighted graphs}

In the previous section we have found necessary and sufficient conditions for $\mm \in \Ass(I^t)$.
The aim of this section is to turn these conditions into a single criterion for $\mm \in \Ass(I^t)$. \par

First, we observe that these conditions involve the existence of a weighted hypergraph $\G_\a$ with the property $\nu(\G_{\a-\e_i}) = \nu(\G_\a)$  for all $i \in \supp(\a)$.  
In graph theory, one calls a graph $\G$ matching-critical if $\nu(\G-i) = \nu(\G)$ for every $i \in V$, where $\G-i$ denotes the subgraph of $\G$ obtained by deleting the vertex $i$. This leads us to  the following weighted version of this notion.

\begin{Definition}
A weighted graph $\G_\a$ is called {\em matching-critical} if $\nu(\G_{\a-\e_i}) = \nu(\G_\a)$ for all $i \in \supp(\a)$.  
\end{Definition}

Now we are going to study properties of matching-critical weighted graphs.

\begin{Lemma}   \label{critical}
$\G_\a$ is a matching-critical weighted graph if and only if $\G^\a$ is a matching-critical graph.
\end{Lemma}

\begin{proof}   
By Lemma \ref{matching},  $\nu(\G_\a) = \nu(\G^\a)$ and $\nu(\G_{\a-\e_i}) = \nu(\G^{\a-\e_i})$. 
By the definition of a parallelization, the graphs $\G^{\a-\e_i}$ and $\G^\a - i_r$ are isomorphic for all $i \in V$ and $r = 1,...,a_i$. Hence, $\G_\a$ is a matching-critical weighted graph if and only if $\nu(\G^\a-i_r) = \nu(\G^\a)$ for all $i \in \supp(\a)$ and $r = 1,...,a_i$. The latter condition just means that $\G^\a$ is a matching-critical graph.
\end{proof}

We say that $\G_\a$ is a {\em connected} weighted graph if the induced subgraph $\G_{\supp(\a)}$ is connected. 
Let $\G_1,...,\G_c$ be the connected components of $\G_{\supp(\a)}$.
Then there exists unique vectors $\a_1,...,\a_c \in \NN^n$ such that $\a = \a_1 + \cdots + \a_c$ and 
$\G_i = \G_{\supp(\a_i)}$, $i = 1,...,c$.
We call the weighted graphs $\G_{\a_1},...,\G_{\a_c}$ the {\em connected components} of $\G_\a$.  
 
\begin{Lemma}  \label{component}
A weighted graph is matching-critical if and only if its connected components are matching-critical.
\end{Lemma}

\begin{proof}   
Let $\G_{\a_1},...,\G_{\a_c}$ be the connected components of $\G_\a$. It is easy to see that 
$$\nu(\G_\a) = \nu(\G_{\a_1}) + \cdots + \nu(\G_{\a_c}).$$
For every $i \in \supp(\a)$, there is a unique $\a_r$ such that $i \in \supp(\a_r)$. Hence
$$\nu(\G_{\a-\e_i}) = \nu(\G_{\a_1}) + \cdots + \nu(\G_{\a_r-\e_i}) + \cdots + \nu(\G_{\a_c}).$$
Consequently, $\nu(\G_\a) = \nu(\G_{\a-\e_i})$ for all $i \in \supp(\a)$ if and only if $\nu(\G_{\a_r}) =  \nu(\G_{\a_r-\e_i})$ for all $i \in \supp(\a_r)$ and $r = 1,...,c$.
\end{proof}

Connected matching-critical graphs have some remarkable characterizations.
Recall that a matching of a graph is called perfect if every vertex is incident to exactly one edge of the matching.
Gallai \cite{Ga} showed that  $\G$ is a matching-critical connected graph
if and only if $\G-i$ has a perfect matching for all $i \in V$.
A graph with this property is called factor-critical \cite{Lo}. 
Gallai's result can be extended to the weighted case as follows.

\begin{Definition}
A matching of a weighted graph $\G_\a$ is called {\em perfect} 
if $a_i$ is the number of times the vertex $i$ appears in the edges of the matching for all $i \in V$. 
\end{Definition}

\begin{Definition}
A weighted graph $\G_\a$ is called {\em factor-critical} 
if $\G_{\a-\e_i}$ has a perfect matching for all $i \in \supp(\a)$.   
\end{Definition}

\begin{Proposition}   \label{Gallai}
A connected weighted graph is matching-critical if and only if it is factor-critical.
\end{Proposition}

\begin{proof}   
As we have seen in the proof of Lemma \ref{matching}, there is a correspondence between matchings of a weighted graph and its parallelization. By this correspondence, a weighted graph $\G_\a$ is factor-critical if and only if  $\G^{\a-\e_i}$ has a perfect matching for all $i \in \supp(\a)$. Since $\G^{\a-\e_i}$ and $\G^\a - i_r$ are isomorphic for all $i \in V$ and $r = 1,...,a_i$, this is equivalent to the condition that $\G^\a$ is factor-critical. By the result of Gallai mentioned above, $\G^\a$ is factor-critical if and only if $\G^\a$ is matching-critical. Hence the assertion follows from Lemma \ref{critical}.
\end{proof}

Let $|\a| = a_1 +\cdots +a_n$. 
It is clear that $|\a|/2$ is an upper bound for the number of edges of a matching of $\G_\a$.
If $\G_\a$ has a perfect matching, this bound is attained and we have $\nu(\G_\a) = |\a|/2$. 

\begin{Corollary}   \label{matching number}
Let $\G_\a$ be a matching-critical weighted graph. Then 
$\nu(\G_\a) = (|\a| - c)/2,$
where $c$ is the number of connected components of $\G$. 
\end{Corollary}

\begin{proof}   
By Lemma \ref{component}, every connected component of $\G_\a$ is matching-critical.
Since the matching number is additive on the connected components, we only need to show the assertion for 
the case $\G_\a$ being connected. In this case, $\G_\a$ is factor-critical by Proposition \ref{Gallai}.  
By definition, $\G_{\a-\e_1}$ has a perfect matching. Hence $\nu(\G_{\a-\e_1}) =  |\a - \e_1|/2$.
Since $\nu(\G_{\a-\e_1}) = \nu(\G_\a)$ and $|\a - \e_1| = |\a|-1$, we obtain $\nu(\G_\a) = (|\a| - 1)/2,$
\end{proof}

The following less known notions in graph theory will play an important role in the characterization of factor-critical weighted graph.

\begin{Definition}
An {\em ear} is a path or a cycle with a specified vertex as endpoints.
An {\em ear decomposition} of a graph is a partition of the edges into a sequence of ears such that the first ear is a cycle and the endpoints of each subsequent ear are the only vertices belonging to earlier ears in the sequence. An ear decomposition is called {\em odd} if all ears have odd lengths. 
\end{Definition}

It is known that a graph has an ear decomposition if and only if it is connected and has no {\it bridge}, 
i.e. an edge whose deletion disconnects the graph \cite{Ro}. \par

By a beautiful result of Lovasz \cite{Lo}, a connected graph with more than one vertex is factor-critical if and only if it has an odd ear decomposition. To prove a similar result for weighted graphs, we need to modify the above notions. 
\par

Recall that a {\em walk} in a graph is a sequence of vertices $i_1,...,i_r$ (repetition possible) such that $\{i_1,i_2\},...,\{i_{r-1},i_r\}$ are edges.  
A walk can be also defined as such a sequence of edges.
The first and the last vertices of the sequence are called the {\em endpoints} of the walk. A walk is {\em closed} if its endpoints coincide.  
The {\em length} of a walk is the number of its edges. 
A walk is called odd or even if its length is odd or even, respectively. 

\begin{Definition}
A {\em generalized ear decomposition} of a weighted graph $\G_\a$ is a sequence $\E$ of walks in $\G$ such that the first walk is closed, the endpoints of each subsequent walk belong to earlier walks, and for all $i \in V$, $a_i$ is the number of times the vertex $i$ appears as inner points of $\E$. Here, we call a vertex of the walks of $\E$  an {\em inner point} of $\E$ if it is neither the last vertex of the first walk nor the endpoints of the subsequent walks 
(we count the first vertex of the first walk as an inner point of $\E$).  
 A generalized ear decomposition is called {\em odd} if all walks are odd. \par
\end{Definition}

For $\a = (1,...,1)$, each vertex of $\G$ appears only once as an inner point of a generalized ear decomposition of $\G_{(1,...,1)}$. 
Using this fact, it is easy to see that a generalized ear decomposition of $\G_{(1,...,1)}$
 is a sequence of walks {\em without repetition of the vertices other than their endpoints} such that the first walk is closed, the endpoints of each subsequent walks are the only vertices belonging to earlier walks in the sequence, and the walks pass through all vertices of $\G$. This is the definition of generalized ear decomposition of $\G$ given in the introduction. 
 \par
 
Since ears are walks or closed walks without repetition of the edges, an ear decomposition of $\G$ is a generalized ear decomposition of $\G$. 
Unlike ear decompositions, the walks of a generalized ear decomposition need not be a partition of the edges of the  graph. It does not need to involve all edges of the graph. In particular, it may 
contain closed walks of length 2 (repetitive edges). 
This allows us to construct generalized ear decompositions in any connected graph. 
For simplicity, we call a closed walk of length 2 a 2-{\em cycle}. 

\begin{Lemma} \label{existence}
A connected graph always has a generalized ear decomposition. 
Moreover, any cycle of the graph can be the first walk of a generalized ear decomposition.
\end{Lemma}

\begin{proof}
We just need to put all edges of the graph into a sequence such that at least a vertex of each subsequent edge belong to earlier edges. 
If an edge has only a vertex which belongs to earlier edges, we consider this edge as a 2-cycle.
If an edge has both vertices belonging to earlier edges, we consider this edge as an 1-path.
By this way, this sequence of edges forms a generalized ear decomposition of the graph.
Similarly, we can construct a generalized ear decomposition which begins with any cycle of the graph.
\end{proof}

\begin{Example} \label{decomposition}
{\rm Consider the graphs of Example \ref{simple} (see Figure 2).
Then $\G$ has no ear decomposition because it has a bridge $\{3,4\}$. 
However, $\G$ has a generalized ear decomposition 
consisting of the triangle $\{1,2,3,1\}$ and the 2-cycle $\{3,4,3\}$. 
Moreover, the weighted graph $\G_\a$  has an odd generalized ear decomposition consisting of only the closed walk $\{1,2,3, 4, 3,1\}$. The parallelization $\G^\a$ has an odd generalized ear decomposition consisting of only the cycle $\{1,2,3_1,4, 3_2,1\}$, which does not involve the edges $\{1,3_1\}$ and $ \{2,3_2\}$ of $\G^\a$. If we add $\{1,3_1\}$ and $ \{2,3_2\}$ to this cycle, we get an odd ear decomposition of $\G^\a$.
}
\end{Example}

With the notion of generalized ear decompositions, Lovasz's characterization of factor-critical graphs can be extended to the weighted case as follows.

\begin{Proposition}   \label{Lovasz}
A connected weighted graph with more than one vertex is factor-critical if and only if it has an odd generalized ear decomposition.
\end{Proposition}

\begin{proof}   
Let $\G_\a$ be a connected weighted graph with more than one vertex.
By Lemma \ref{critical} and Proposition \ref{Gallai},  
$\G_\a$ is factor-critical if and only if $\G^\a$ is factor-critical. 
By the afore mentioned result of Lovasz, $\G^\a$ is factor-critical if and only if it has an odd ear decomposition. 
Therefore, it suffices to show that $\G^\a$ has an odd ear decomposition if and only if $\G_\a$ has an odd generalized  ear decomposition. \par
 
Assume that $\G^\a$ has an odd ear decomposition $\E$. 
By the projection of $\G^\a$ to $\G$ we obtain from every ear of $\E$ a walk in $\G$ of the same length. 
Let $\F$ be the sequence of these walks.
It is clear that the first walk is closed and that the endpoints of each subsequent walk belong to earlier walks.
By the definition of ear decomposition, every vertex $i_r$ of $\G^\a$ appears only once as an an inner point of $\E$.
From this it follows that $a_i$ is the number of the inner points of the form $i_r$ of $\E$.
Since every inner point $i$ of $\F$ is the image of an inner point $i_r$ of $\E$,  
$a_i$ is the number of times the vertex $i$ appears as an inner point of $\F$. 
Therefore, $\F$ is an odd generalized ear decomposition of $\G_\a$. \par

Conversely, assume that $\G_\a$ has an odd generalized ear decomposition $\F$. We associate to every walk of $\F$ an odd ear of $\G^\a$ such that its projection to $\G$ is the given walk.  
Let $\E$ be the sequence of these odd ears. 
Obviously, the first ear is a cycle and the endpoints of subsequent ears belong to earlier ears. 
Note that for every vertex $i \in V$, there are $a_i$ vertices of $\G^\a$ which can be projected to $i$.
Since $a_i$ is the number of times the vertex $i$ appears as an inner point of $\F$,
we can choose the ears of $\E$ such that every vertex of $\G^\a$ appears only once as an inner point of $\E$. 
From this it follows that the endpoints of each subsequent ear are the only vertices belonging to earlier ears in $\E$.  \par

As a consequence, the edges of an ear of $\E$ are different. However, an edge of $\G^\a$ may belong 
to two different ears of $\E$. In this case, the later ear consists of only this edge because the vertices of this edge are the only vertices of that ear belonging to earlier ears. Therefore, one can remove this ear from $\E$ to obtain a sequence of ears with similar properties. By this way, we get a new sequence $\E'$ of ears whose edges are all different.
There may exist edges of $\G^\a$ which are not contained in any ear of $\E'$. Since the ears of $\E'$ involve all vertices of $\G^\a$, the vertices of all missing edges belong to the ears of $\E'$. Therefore,
we can add all missing edges as paths of length one to $\E'$ to obtain an ear decomposition of $\G^\a$. 
Obviously, this ear decomposition of $\G^\a$ is odd.
\end{proof}

As an application of the above weighted versions of the results of Gallai and Lovasz  we show how to construct from a matching-critical weighted graph a matching-critical weighted graph of any higher matching number on the same vertex set.

\begin{Corollary} \label{add}
Let $\G_\a$ be a matching-critical weighted graph.
Let $\{i,j\}$ be an arbitrary edge of $\G_{\supp(\a)}$. 
Then $\G_{\a+\e_i+\e_j}$ is a matching-critical weighted graph with $\nu(\G_{\a+\e_i+\e_j}) = \nu(\G_\a)+1$. 
\end{Corollary}

\begin{proof} 
By Lemma \ref{component}, we may assume that $\G_\a$ is connected. Then $\G_\a$ is factor-critical by Proposition \ref{Gallai}.  
By the proof of the necessary part of Proposition \ref{Lovasz}, $\G_\a$ has an odd generalized ear decomposition $\E$
which is the projection of an odd ear decomposition of $\G^\a$. Therefore, $\E$ involves all edges of $\G_{\supp(\a)}$.
Let $C$ be an odd walk of $\E$ containing the edge $\{i,j\}$.
Let $D$ be the odd walk obtained from $C$ by replacing $\{i,j\}$ by the sequence $\{i,j\}, \{j,i\}, \{i,j\}$. 
Replacing $C$ by $D$ we obtain from $\E$ an odd generalized ear decomposition of $\G_{\a+\e_i+\e_j}$.
Hence, $\G_{\a+\e_i+\e_j}$ is matching-critical by Proposition \ref{Gallai} and Proposition \ref{Lovasz}. 
By Corollary \ref{matching number}, we have 
$$\nu(\G_{\a+\e_i+\e_j}) = (|\a+\e_i+\e_j|-1)/2 = (|\a|+1)/2 = \nu(\G_\a)+1.$$
\end{proof}

Now we can give a criterion for $\mm \in \Ass(I^t)$ by means of matching-critical weighted graphs.

\begin{Theorem} \label{max 2} 
$\mm \in \Ass(I^t)$ if and only if there exists a matching-critical weighted graph $\G_\a$ without isolated vertices such that $\supp(\a)$ is a dominating set of $\G$ and $\nu(\G_\a) < t$.
\end{Theorem}

\begin{proof}   
Since $I$ is a radical ideal with $\dim R/I \ge 1$, $\mm \not\in \Ass(I)$.
Assume that $\mm \in \Ass(I^t)$, $t \ge 2$. Let $s$ be the largest integer 
 $< t$ such that $\mm \not\in \Ass(I^s)$.  Then $\mm \in \Ass(I^{s+1}) \setminus \Ass(I^s)$.
By Proposition \ref{max 1}, this condition implies the existence of a matching-critical weighted graph $\G_\a$ without isolated vertices such that $\supp(\a)$ is a dominating set of $\G$ and $\nu(\G_\a) = s$. \par
Conversely, assume that  there exists a matching-critical weighted graph $\G_\a$ without isolated vertices such that $\supp(\a)$ is a dominating set of $\G$ and $\nu(\G_\a) < t$. Using Corollary \ref{add},  we can construct a matching-critical weighted graph $\G_\b$ without isolated vertices such that $\supp(\b) = \supp(\a)$  and $\nu(\G_\b) = t-1$. By Proposition \ref{socle 2}, this implies $\mm \in \Ass(I^t)$.
\end{proof}

\begin{Remark}
{\rm All notions of Theorem \ref{max 2} can be extended in a straightforward way to hypergraphs.  
However, we could not prove a similar criterion for $\mm \in \Ass(I^t)$ when $I$ is the edge ideal of a hypergraph.
}
\end{Remark}


\section{Base graphs}

We have obtained with Theorem \ref{max 2} a combinatorial criterion for $\mm \in \Ass(I^t)$.
However, the criterion relies on the existence of a matching-critical weighted graph $\G_\a$ with additional properties.
The existence of such $\G_\a$ can not be checked effectively. 
The aim of this section is to deduce a criterion for $\mm \in \Ass(I^t)$ solely in terms of $\G$. 

First, we investigate the problem what impact a matching-critical weighted graph $\G_\a$ may have on the induced graph $\G_{\supp(\a)}$. For simplicity, we call $\G_{\supp(\a)}$ the {\em base graph} of $\G_\a$. 
\par

By Proposition \ref{Gallai} and Proposition \ref{Lovasz}, a connected weighted graph $\G_\a$ with more than one vertex is matching-critical if and only it has an odd generalized ear decomposition. Since the graph $\G_{\supp(\a)}$ can be obtained from $\G_\a$ by lowering the weights one by one, we want to see what happen with a generalized ear decomposition in passing from $\G_\a$ to $\G_{\a-\e_i}$ for $i \in \supp(\a)$ with $a_i \ge 2$. \par
 
For a sequence $\E$ of walks, we denote by $\varphi(\E)$ the number of even walks in $\E$.

\begin{Lemma}   \label{reduction} 
Let $\G_\a$ be a connected weighted graph and $\E$ a generalized ear decomposition of $\G_\a$. 
Let $i \in \supp(\a)$ with $a_i \ge 2$. By breaking suitable walks of $\E$ we can turn $\E$ into a generalized ear decomposition $\F$ of $\G_{\a - \e_i}$ such that \par
{\rm (i) } $\varphi(\F) \le \varphi(\E)+1$,\par
{\rm(ii) } If the first walk of $\E$ is odd, then so is the first walk of $\F$.
\end{Lemma}

\begin{proof}   
Let $C$ be the earliest walk of $\E$ containing $i$ as an inner point.  
We distinguish the following cases. \par

{\em Case 1}: $C$ passes the vertex $i$ as an inner point once. \par
 
Since $a_i \ge 2$, there exists another walk of $\E$ containing $i$ as an inner point.
Let $C'$ be the next walk of $\E$ containing $i$ as an inner point. 
Let $u,v$ be the endpoints of $C'$ (which may be the same). 
We break $C'$ at an inner point $i$ of $C$ into a walk $D$ from $u$ to $i$ and a walk $D'$ from $i$ to $v$ (or to $u$ if $u = v$). See Figure 3.  Replacing $C'$ by $D, D'$ we obtain from $\E$ a new sequence $\F$ of walks.
The breaking inner point $i$ of $C'$ is no more an inner point of $\F$, while other inner points in $\F$ are the same as in $\E$. Therefore, $a_i-1$ is the number of times the vertex $i$ appears as inner points of $\F$ and $a_j$, $j \neq i$, is the number of times the vertex $j$ appears as inner points of $\F$. 
It is now clear from our construction that $\F$ is a generalized ear decomposition of $\G_{\a-\e_i}$.  

\begin{figure}[ht!]
\begin{tikzpicture}[scale=0.6] 

\draw [thick] (0,0) coordinate (a) -- (2,0.7) coordinate (b) ;
\draw [thick] (2,0.7) coordinate (b) -- (3,2) coordinate (c) ;
\draw [thick] (3,2) coordinate (c) -- (2,3.4) coordinate (d) ;
\draw [thick] (2,3.4) coordinate (d) -- (0,4) coordinate (e) ;
\draw [thick] (3,2) coordinate (c) -- (4,1) coordinate (x) ;
\draw [thick] (4,1) coordinate (x) -- (3,0) coordinate (y) ;
\draw [thick] (3,2) coordinate (c) -- (4,3) coordinate (z) ;
\draw [thick] (4,3) coordinate (z) -- (3,4) coordinate (t) ;

\fill (c) circle (3pt);
\fill (y) circle (3pt);
\fill (t) circle (3pt);
\fill (a) circle (3pt);
\fill (e) circle (3pt);

\draw (1.5,2) node{$C$};
\draw (3.5,2) node[right] {$C'$};
\draw (3,2) node[left] {$i$};
\draw (3,0) node[below] {$v$};
\draw (3.6,0.5) node[right] {$D'$};
\draw (3.6,3.6) node[right] {$D$};
\draw (3,4) node[above] {$u$};

\draw (2,-1.1) node[below] {\it Case 1};

\draw [thick] (11,0) coordinate (v) -- (13,0.7) coordinate (h) ;
\draw [thick] (13,0.7) coordinate (h) -- (14,2) coordinate (i) ;
\draw [thick] (14,2) coordinate (i) -- (13,3.4) coordinate (j) ;
\draw [thick] (13,3.4) coordinate (j) -- (11,4) coordinate (u) ;
\draw [thick] (14,2) coordinate (i) -- (15,1) coordinate (k) ;
\draw [thick] (15,1) coordinate (k) -- (16.5,2) coordinate (l) ;
\draw [thick] (16.5,2) coordinate (l) -- (15,3) coordinate (m) ;
\draw [thick] (15,3) coordinate (m) -- (14,2) coordinate (i) ;

\fill (v) circle (3pt);
\fill (u) circle (3pt);
\fill (i) circle (3pt);

\draw (12.7,2) node{$C$};
\draw (11,0) node[below] {$v$};
\draw (14,2) node[left] {$i$};
\draw (11,4) node[above] {$u$};
\draw (13.7,3) node[above]{$D$};
\draw (15.8,1.4) node[below] {$D'$};

\draw (14,-1.1) node[below] {\it Case 2};

\end{tikzpicture}
\caption{}
\end{figure}

Note that the number of the edges of $C'$ is the sum of the numbers of the edges of $D$ and $D'$.
If $C'$ is an even walk, then both $D$ and $D'$ are odd or even walks. 
Since $\E$ and $\F$ share all other walks, we get $\varphi(\F) \le \varphi(\E)+1$.
If $C'$ is an odd walk, then one of the walks $D$ or $D'$ is odd and the other even. 
In this case, we have $\varphi(\F) = \varphi(\E)+1$.

Since $C'$ is not the first walk of $\E$, the sequences $\E$ and $\F$ share the same first walk.

{\em Case 2}:  $C$ passes the vertex $i$ as an inner point at least twice. \par

Let $u,v$ be  the endpoints of $C$ (which may be the same).
Let $D$ be the walk in $C$ which goes from $u$ to $i$, where $i$ appears the first time as an inner point in $C$, and then from $i$, where $i$ appears the last time as an inner point in $C$, to $v$ (or to $u$ if $u = v$); see Figure 3. 
Let $D'$ be the remaining closed walk in $C$.
Replacing $C$ by $D, D'$ we obtain from $\E$ a new sequence $\F$ of walks.
It is clear $i$ appears as inner point in $\F$ one time less than in $\E$, while other inner points in $\F$ are the same as in $\E$. Similarly as in Case 1, we can show that $\F$ is a generalized ear decomposition of $\G_{\a-\e_i}$ with
$\varphi(\F) \le \varphi(\E)+1$. \par

If $C$ is not the first walk of $\E$, the sequences $\E$ and $\F$ share the same first walk.
If $C$ is the first walk of $\E$, then $D$ is the first walk of $\F$. 
Since $C$ is a closed walk, $u = v$. Hence, $D$ is a closed walk.
If $C$ is odd, one of the closed walks $D$ and $D'$ is odd. 
If $D$ is not odd, we replace the sequence $D, D'$ by $D', D$ in the construction of $\F$. 
Then $\F$ is still a generalized ear decomposition of $\G_{\a-\e_i}$ with $\varphi(\F) \le \varphi(\E)+1$, whose  
first walk $D'$ is odd.\par

Thus, if the first walk of $\E$ is odd, then so is the first walk of $\F$.
\end{proof}

Let us now extend the notion of generalized ear decomposition to the case where $\G$ may be a disconnected graph. 

\begin{Definition}
We call a family of generalized ear decompositions of the connected components of $\G$ a {\em generalized ear decomposition} of $\G$. If the generalized ear decompositions of such a family begin with odd cycles, we call it an {\em odd-beginning generalized ear decomposition}. 
\end{Definition}

Recall that $\G$ is called a {\em strongly non-bipartite} graph if every connected component is non-bipartite or, equivalently, if every connected component contains an odd cycle. 
By Lemma \ref{existence}, every cycle of a connected graph can be chosen to be the first walk of a generalized ear decomposition. Therefore, a strongly non-bipartite graph $\G$ always has an odd-beginning generalized ear decomposition, and we can introduce the following invariants.

\begin{Definition}
We denote by $\varphi^*(\G)$ the minimal total number of even walks in odd-beginning generalized ear decompositions of $\G$ and define
$$\mu^*(\G) := (\varphi^*(\G) + n - c)/2,$$ 
where $c$ is the number of connected components of $\G$. 
\end{Definition}
 
\begin{Example} \label{odd cycle}
Let $\G$ be an odd cycle of length $2t+1$. 
The cycle forms an odd-beginning generalized ear decomposition of $\G$. 
Hence, $\varphi^*(\G) = 0$ and $\mu^*(\G) = t$. 
\end{Example}

Using the notion $\mu^*(\G)$ we can give a characterization of the base graph of a matching-critical weighted graph as follows. 

\begin{Theorem}  \label{base}
Let $\G$ be a graph without isolated vertices. Then
$\G$ is the base graph of a matching-critical weighted graph $\G_\a$ with $\nu(\G_\a) = t$ if and only if $\G$ is strongly non-bipartite with $\mu^*(\G) \le t$.
\end{Theorem}

\begin{proof}   
It is easy to see that the invariant $\mu^*(\G)$ is additive on the connected components of $\G$.
Therefore, using Lemma \ref{component} we only need to prove the assertion for a connected graph $\G$
with more than one vertex. \par
 
Assume that $\G$ is the base graph of a matching-critical weighted graph $\G_\a$ with $\nu(\G_\a) = t$.
 By Proposition \ref{Gallai}, $\G_\a$ is factor-critical.
Hence, $\G_\a$ has an odd ear decomposition $\E$, i.e.  $\varphi(\E) = 0$, by Proposition \ref{Lovasz}. 
Since $\G$ is the base graph of $\G_\a$, $\supp(\a) = V$.  Hence $a_i \ge 1$ for all $i = 1,...,n$. Since $\G = \G_{(1,...,1)}$ and since $(1,...,1)$ can be obtained from $\a$ by $|\a| - n$ subtractions by unit vectors,  we can apply Lemma \ref{reduction} successively to obtain from $\E$ an odd-beginning generalized ear decomposition $\F$ of $\G$ with 
$$\varphi(\F) \le \varphi(\E) + |\a| - n = |\a|-n.$$
This implies $\varphi^*(\G) \le |\a| - n$.
Therefore, $\G$ is a non-bipartite graph with
$$\mu^*(\G) = (\varphi(\G) +n - 1)/2 \le (|\a| - 1)/2.$$
By Corollary \ref{matching number}, we have $(|\a| - 1)/2 = \nu(\G_\a) = t$. Hence, $\mu^*(\G) \le t$.
 \par
   
Conversely, assume that $\G$ is a non-bipartite graph with $\mu^*(\G) \le t$.  
By Corollary \ref{add}, we only need to show that $\G$ is the base graph of a matching-critical weighted graph $\G_\a$ with $\nu(\G_\a) \le t$. \par

Let $\E$ be an odd-beginning generalized ear decomposition of $\G$ with $\varphi(\E) = \varphi^*(\G)$. 
Let $C$ be an arbitrary even walk of $\E$.
Then $C$ is not the first walk of $\E$. 
Let $i$ be the first vertex of $C$ and $v$ an adjacent vertex of $i$ in an earlier walk of $\E$.
Adding the edge $\{v,i\}$ to $C$ at the beginning we obtain from $C$ a walk $C'$ of length one more than that of $C$.
If we replace every even walk $C$ of $\E$ by such a walk $C'$, we obtain from $\E$ a sequence $\F$ of odd walks. \par

It is obvious that the first walk of $\F$ is the first walk of $\E$ and that the endpoints of each subsequent walk of $\F$ belong to earlier walks. 
By our construction, we have only a new inner point $i$ every time we replace an even walk $C$ by $C'$, where $i$ is the starting vertex of $C$. Other inner points of $C$ remains the same in $C'$.
By the definition of generalized ear decomposition, every vertex $i \in V$ appears only once as an inner point of $\E$.
Therefore, if $c_i$ is the number of even walks of $\E$ whose starting vertex is $i$, then
$c_i+1$ is the number of times the vertex $i$ appears as inner points of $\F$.
Define the vector $\a \in \NN^n$ by setting $a_i = c_i+1$ for all $i \in V$. Then
$\F$ is an odd generalized ear decomposition of $\G_\a$. \par
 
So $\G_\a$ is a matching-critical weighted graph by Proposition \ref{Gallai} and Proposition \ref{Lovasz}.
By Corollary \ref{matching number}, 
$$\nu(\G_\a) = (|\a| -1)/2 = (\sum_{i \in V}c_i + n - 1)/2.$$
On the hand, we have $\sum_{i \in V}c_i = \varphi(\E) = \varphi^*(\G).$ Hence,
$$\nu(\G_\a) = (\varphi^*(\G) +n-1)/2 = \mu^*(\G) \le t,$$
as required. 
\end{proof}

Now we are able to give a criterion for $\mm \in \Ass(I^t)$ in terms of $\G$ without involving a weighted graph. 

\begin{Theorem} \label{maximal} 
$\mm \in \Ass(I^t)$ if and only if there exists a dominating set $U$ of $\G$ such that $\G_U$ is strongly non-bipartite with $\mu^*(\G_U) < t$.
\end{Theorem}

\begin{proof}   
By Theorem \ref{max 2}, $\mm \in \Ass(I^t)$ if and only if there exists a matching-critical weighted graph $\G_\a$ without isolated vertices such that $\supp(\a)$ is a dominating set of $\G$ and $\nu(\G_\a) < t$. By Theorem \ref{base}, this condition is satisfied if and only if $\G_{\supp(\a)}$ is strongly non-bipartite with $\mu^*(\G_{\supp(\a)}) < t$.
\end{proof}

From Theorem \ref{maximal} we immediately obtain the following criterion for $\mm \in \Ass^\infty(I)$ (see \cite[Lemma 2.1 and Corollary 3.4]{CMS} or \cite[Corollary 3.6]{HHT}). 

\begin{Corollary}  \label{maximal 2}
$\mm \in \Ass(I^t)$ for some $t \ge 1$ (or for all $t \gg  0$) if and only if $\G$ is a strongly non-bipartite graph.
\end{Corollary}

We even have an exact formula for the least number $t$ such that $\mm \in \Ass(I^t)$. 
This formula can not be deduced from \cite{CMS} or \cite{HHT}.

\begin{Corollary} \label{least}
Let $\G$ be a strongly non-bipartite graph. Let $s(\G)$ denote the minimum of $\mu^*(\G_U)$, where $U$ is a dominating set of $\G$ such that $\G_U$ is strongly non-bipartite. 
Then $\mm \in \Ass(I^t)$ if and only if $t \ge s(\G)+1$.
\end{Corollary}

By definition, we always have $s(\G) \le \mu^*(\G)$.
The following lemma shows that the difference $\mu^*(G) - s(\G)$ is no less than the number of leaf vertices, i.e. vertices of degree one of $\G$. 

\begin{Lemma} \label{leaf}
Let $\G$ be a strongly non-bipartite graph. Let $W$ be a set of leaf vertices of $\G$.  
Then $V \setminus W$ is a dominating set of $\G$ and the induced subgraph $\G_{V \setminus W}$ is a strongly non-bipartite graph with $\mu^*(\G) = \mu^*(\G_{V \setminus W})+|W|$.
\end{Lemma}

\begin{proof}   
Since every connected component of $\G$ is not an edge, each leaf vertex is adjacent to a vertex of degree at least two. 
Hence, $V \setminus W$ is a dominating set of $\G$. It is also clear that $\G_{V \setminus W}$ contains all odd cycles of $\G$. Hence, $\G_{V \setminus W}$ is a strongly non-bipartite graph. \par

To show that $\mu^*(\G) = \mu^*(\G_{V \setminus W})+|W|$, we only need to show that $\mu^*(\G) = \mu^*(\G-i)+1$ for any leaf vertex $i$. Let $j$ be the vertex adjacent to $i$. Then every generalized ear decomposition of $\G$ must contain the 2-cycle of the edge $\{i,j\}$. From this it follows that $\varphi^*(\G) = \varphi^*(\G-i)+1$, which implies the assertion.
\end{proof}

In the following example we show that $\mu^*(\G) - s(\G)$ can be arbitrary large even if $\G$ has no leaf vertex. 

\begin{Example} \label{flower} 
Let $\G$ be the graph on $2t+1$ vertices which consists of $t$ triangles having a common vertex (see Figure 4 for the case $t = 2$).

\begin{figure}[ht!]

\begin{tikzpicture}[scale=0.6] 
        
\draw [thick] (2,0) coordinate (b) -- (2,2) coordinate (a) ;
\draw [thick] (2,2) coordinate (a) -- (3.5,1) coordinate (c);
\draw [thick] (3.5,1) coordinate (c) -- (2,0) coordinate (b);
\draw [thick] (3.5,1) coordinate (c) -- (5,2) coordinate (d);
\draw [thick] (3.5,1) coordinate (c) -- (5,0) coordinate (e);
\draw [thick] (5,0) coordinate (e) -- (5,2) coordinate (d);
    
\fill (a) circle (3pt);
  \fill (c) circle (3pt);
  \fill (d) circle (3pt);
  \fill (e) circle (3pt);
  \fill (b) circle (3pt);
  
\end{tikzpicture}
\caption{}
\end{figure}

It is clear that the triangles form an odd ear decomposition of $\G$. Hence, $\varphi^*(\G) = 0$. This implies 
$$\mu^*(\G) = (\varphi^*(\G) + 2t)/2 = t.$$
Let $U$ be the vertices of a triangle of $\G$. Then $U$ is a dominating set of $\G$ and $\G_U$ is this triangle. Since $\mu^*(\G_U) = 1$, the smallest possible value for the $\mu^*$-invariant of a graph, we have $s(\G) = 1$.
\end{Example}

\begin{Remark}
The invariant $\mu^*(\G)$ is inspired by the invariant $\mu(\G)$ of a connected graph $\G$ having no bridge, which was introduced by Sole and Zaslavsky in coding theory \cite{SZ}.  Frank \cite{Fr} showed that 
$$\mu(\G) =  (\varphi(\G)+n-1)/2,$$
where $\varphi(\G)$ denotes the minimal number of even ears in ear decompositions of $\G$. 
We do not know whether $\varphi^*(\G) = \varphi(\G)$ for every non-bipartite graph $\G$ having no bridge.
\end{Remark}

 
\section{Associated primes of powers of edge ideals}

For a subset $F$ of $V$ we set $P_F = (x_i|\ i \in F)$.
It is well known that every associated prime of $I^t$ is of the form $P_F$, where $F$ is a (vertex) cover of $\G$, i.e. $F$ meets every edge of $\G$. In particular, $P_F$ is a minimal associated prime of $I^t$ if and only if $F$ is a minimal cover of $\G$ (see e.g. \cite{HHT}). The remaining problem is to find a combinatorial criterion for 
$P_F$ to be an embedded prime of $I^t$ in terms of $\G$. 
The aim of this section is to deduce from Theorem \ref{maximal} such a criterion. \par

The above problem can be reduced to the case $P_F = \mm$ by using the following notion from \cite{HLT}. 

\begin{Definition}
Let $c(F)$ denote the set of vertices of $F$ which are not adjacent to any vertex of $V \setminus F$.
We call $c(F)$ the {\em core} of $F$. 
\end{Definition}

It is easy to check that $c(F) = \emptyset$ if and only if $F$ is a minimal cover of $\G$.

\begin{Lemma} \label{core}  
Let $F$ be a cover of $\G$. Let $S = k[x_i|\ i \in c(F)]$. 
Let $\nn$ be the maximal homogeneous ideal of $S$ and $J$ the edge ideal of the induced subgraph $\G_{c(F)}$ in $S$.
Then $P_F \in \Ass(I^t)$ if and only if $\nn \in \Ass(J^s)$ for some $s \le t$.
\end{Lemma}

\begin{proof}
Let $A = R[x_j^{-1}|\ j \in V \setminus F]$. 
Since $A$ is a localization of $R$, $P_F$ is an associated prime of $I^t$ if and only if $P_FA$ is an associated prime of $I^tA$. 
If  $i \in F \setminus c(F)$,  then $i$ is adjacent to a vertex $j \in V \setminus F$. 
Hence $x_i = (x_ix_j)x_j^{-1}\in IA$. 
From this it follows that  $IA$ is the ideal generated by $J$ and the variables $x_i$, $i \in F \setminus c(F)$. \par
 
Let $B = k[x_i|\ i \in F]$. Let $\pp$ be the maximal homogeneous ideal of $B$ and $Q \subset B$ the ideal generated by $J$ and the variables $x_i$, $i \in F \setminus c(F)$.
Then $P_FA = \pp A$ and $I^tA = Q^tA$.   
Since $A = B[x_j^{\pm 1}|\ j \in V \setminus F]$ is a Laurent polynomial ring over $B$, $P_FA$ is an associated prime of $I^t A$ if and only if $\pp$ is an associated prime of $Q^t$.  
Note that $\pp$ and $Q$ are the extensions of the ideals $\nn$ and $J$ in $S$ by the variables $x_i$, $i \in F \setminus c(F)$, in 
the polynomial ring $B = S[x_i|\ i \in F \setminus c(F)]$. Then we can apply \cite[Lemma 3.4]{HM}, which shows that $\pp$ is an associated prime of $Q^t$ if and only if $\nn$ is an associated prime of $J^s$ for some $s \le t$.
\end{proof}

By Theorem \ref{maximal}, to check the condition $\nn \in \Ass(J^s)$ we have to deal with dominating sets of $\G_{c(F)}$. Such a set can be described in terms of $\G$ as follows.

\begin{Lemma} \label{dominating}  
Let $F$ be a cover of $\G$. A subset $U \subseteq V$ is a dominating set of $\G_{c(F)}$ if and only if $F$ is minimal among the covers containing $N[U]$.
\end{Lemma}

\begin{proof} 
By definition, $U$ is a dominating set of $\G_{c(F)}$ if and only if $U \subseteq c(F)$ and $c(F) \subseteq N[U]$.  
It is clear that  $U \subseteq c(F)$ if and only if every vertex of $U$ is not adjacent to any vertex of $V \setminus F$ if and only if $F \supseteq N[U]$. Now we may assume that $F \supseteq N[U]$.
Then $c(F) \subseteq N[U]$ if and only if $F \setminus N[U] \subseteq F \setminus c(F)$, 
which is the set of the vertices $i \in F$ adjacent to a vertex of $V \setminus F$. 
This condition means $F \setminus i$ is not a cover of $\G$. Hence, $c(F) \subseteq N[U]$ if and only if $F$ is minimal among the covers containing $N[U]$.
Therefore, $U$ is a dominating set of $\G_{c(F)}$ if and only if $F$ is minimal among the covers containing $N[U]$.
\end{proof}

Now we are able to prove the main result of this paper, which is a characterization of the associated primes of $I^t$ in terms of $\G$.

\begin{Theorem} \label{associated}
Let $F$ be a cover of $\G$. Then $P_F$ is an associated prime of $I^t$ if and only if $F$ is either a minimal cover or minimal among the covers containing $N[U]$ for some subset $U \subseteq V$ such that $\G_U$ is strongly non-bipartite with $\mu^*(\G_U) < t$. 
\end{Theorem}

\begin{proof}
We may assume that $F$ is not a minimal cover. Then $c(F) \neq \emptyset$.
Let $S = k[x_i|\ i \in c(F)]$. Let $\nn$ be the maximal homogeneous ideal of $S$ and $J$ the edge ideal of the induced subgraph $\G_{c(F)}$ in $S$.
By Lemma \ref{core}, $P_F \in \Ass(I^t)$ if and only if $\nn  \in \Ass(J^s)$ for some $s \le t$. By Theorem \ref{maximal}, this is equivalent to the condition that there is a dominating set $U$ of $\G_{c(F)}$ such that $\G_U$ is a strongly non-bipartite graph with $\mu^*(\G_U) \le t-1$. 
Now we only need to apply Lemma \ref{dominating} to obtain the assertion.
\end{proof}

By Theorem \ref{associated}, $P_F$ is an embedded prime of $I^t$ if and only if $F$ is minimal among the covers containing $N[U]$ for some subset $U \subseteq V$ such that $\G_U$ is strongly non-bipartite with $\mu^*(\G_U) < t$. 
As a consequence, $I^t$ has no embedded prime for all $t \ge 1$ if and only if $\G$ has no odd cycles or, equivalently, $\G$ is a bipartite graph.  
This is a celebrated result of Simis, Vasconcelos, and Villarreal \cite[Theorem 5.9]{SVV}.
Furthermore, we can easily deduce the following criterion for the existence of embedded primes of $I^t$, 
which was given by Rinaldo, Terai and Yoshida \cite[Lemma 3.10]{RTY}.  

\begin{Corollary} \label{SVV}
$I^t$ has no embedded prime if and only if $\G$ has no odd cycle of length $\le 2t-1$.
\end{Corollary}

\begin{proof}    
By Theorem \ref{associated}, $I^t$ has an embedded prime if and only if there exists a strongly non-bipartite induced subgraph $\G_U$ with $\mu^*(\G_U) \le t-1$. 
If $I^t$ has an embedded prime, such a subgraph $\G_U$ has an odd cycle $C$. If $C$ has length $2s-1$, we have
$$2s - 1\le |U| \le  2\mu^*(\G_U) +1 \le 2t - 1.$$
Conversely, if $\G$ has an odd cycle $C$ of length $\le 2t-1$, we let $U$ to be the vertex set of $C$.
Then $C$ is an odd-beginning generalized ear decomposition (of one walk) of $\G_U$. From this it follows that  $\varphi^*(\G_U) = 0$.
Hence, 
$$\mu^*(\G_U) = (\varphi^*(\G_U) + |U| +1)/2 \le t-1.$$
By Theorem \ref{associated}, this implies that $I^t$ has an embedded prime.
\end{proof}

It is clear from Theorem \ref{associated} that if $P_F \in \Ass(I^t)$, then $P_F \in \Ass(I^{t+1})$. Hence,
we immediately obtain the following result of Martinez-Bernal, Morey and Villarreal  \cite[Theorem 2.15]{MMV}.

\begin{Corollary} 
$\Ass(I^t) \subseteq \Ass(I^{t+1})$ for all $t \ge 1$. 
\end{Corollary}

Let $\Ass^\infty(I)$ denote the stable set $\Ass(I^t)$ for $t \gg 0$.
Another immediate consequence of Theorem \ref{associated} is that $P_F \in \Ass^\infty(I)$ if and only if $F$ is either a minimal cover or minimal among the covers containing $N[U]$ for some subset $U \subseteq V$ such that $\G_U$ is strongly non-bipartite. This explicit description of $\Ass^\infty(I)$ was obtained recently by Hien, Lam and Trung \cite[Corollary 3.6]{HLT}. If $\G$ is a connected graph, it can be also deduced from an algorithmic construction of 
$\Ass^\infty(I)$ by Chen, Morey and Sung \cite{CMS}. We give below a simpler description of $\Ass^\infty(I)$. \par

Let $\A(\G)$ denote the set of the minimal covers and the non-minimal covers $F$ of $\G$ such that $\G_{c(F)}$ is a strongly non-bipartite graph.

\begin{Corollary} \label{infty}
$P_F \in \Ass^\infty(I)$ if and only if $F \in \A(\G)$. 
\end{Corollary}

\begin{proof}
We may assume that $F$ is not a minimal cover. 
If $P_F \in \Ass^\infty(I)$, then $F$ is minimal among the covers containing $N[U]$ for some subset $U \subseteq V$ such that $\G_U$ is strongly non-bipartite.
By Lemma \ref{dominating}, we have $U \subseteq c(F) \subseteq N[U]$. 
Hence, every connected component of $\G_{c(F)}$ contains a connected component of $\G_U$. Therefore, $\G_{c(F)}$  is also strongly non-bipartite.
Conversely, if $F \in \A(\G)$, we let $U = c(F)$. Then $F$ is a minimal cover of $N[U]$ by Lemma \ref{dominating}. Since $\G_U$ is strongly non-bipartite, $F \in \Ass^\infty(I)$.
\end{proof}

Sharp \cite{Sh} asked the following question (in a more general setting): Given a prime ideal $P_F \in \Ass^\infty(I)$, can one identify an integer $s$ such that $P_F \in \Ass(I^t)$ for all $t \ge s$.  
We can give a positive answer to this question by using the notion $s(\G)$ introduced in Corollary \ref{least}.\par

Note that $c(F) = \emptyset$ if $F$ is a minimal cover of $\G$. In this case, we set $s(\G_{c(F)}) := 0$.
If $F$ is a not a minimal cover of $\G$, $c(F) \neq \emptyset$. In this case, $s(\G_{c(F)})$ is the minimum of $\mu^*(\G_U)$, where $U$ is a dominating set of $\G_{c(F)}$ such that $\G_U$ is strongly non-bipartite.  

\begin{Theorem}  \label{stab}
Let $P_F \in \Ass^\infty(I)$. 
Then $P_F \in \Ass(I^t)$ if and only if $t > s(\G_{c(F)})$. 
\end{Theorem}

\begin{proof} 
If $F$ is a minimal cover of $\G$, $P_F$ is a minimal prime of $I$. Hence, $P_F \in \Ass(I^t)$ for $t > 0$.
Assume that $F$ is not a minimal cover of $\G$.
By Theorem \ref{associated},  $P_F \in \Ass(I^t)$ if and only if  $t > \mu^*(\G_U)$ for a set $U \subseteq V$ such that 
$F$ is minimal among the covers containing $N[U]$ and $\G_U$ is strongly non-bipartite. Let $s$ be the minimum of 
such $\mu^*(\G_U)$. Then $P_F \in \Ass(I^t)$ if and only if $t > s$.
By Lemma \ref{dominating},  $F$ is minimal among the covers containing $N[U]$ if and only if $U$ is a dominating set of $\G_{c(F)}$. Therefore, $s = s(\G_{c(F)})$.
\end{proof}

From Corollary \ref{infty} and Theorem \ref{stab} we immediately obtain the following precise formula for $\astab(I)$, which is defined as the least number $t_0$ such that $\Ass(I^t) = \Ass(I^{t+1})$ for $t \ge t_0$.  

\begin{Corollary} \label{astab}
$\astab(I) = \max\{s(\G_{c(F)})|\  F \in \A(\G)\} +1.$
\end{Corollary}

We can also deduce the following simple upper bound  for $\astab(I)$.
Let $\B(\G)$ denote the sets of all subsets $U \subseteq V$ such that $\G_U$ is strongly non-bipartite and has no leaf vertices. 

\begin{Corollary} \label{astab 1}
Let $\G$ be a non-bipartite graph. Then
$$\astab(I) \le \max \{\mu^*(\G_U)|\ U \in \B(\G)\}+1.$$
\end{Corollary}

\begin{proof} 
We only need to show that for all $P_F \in \Ass^\infty(I)$, there exists $U \in \B(\G)$ such that $P_F \in P^t$ if $t \ge \mu^*(\G_U)+1$. We may assume that $F$ is not a minimal cover. Then $\G_{c(F)}$ is strongly non-bipartite.
Let $c'(F)$ denote the set of the vertices of degree $\ge 2$ of $\G_{c(F)}$.
Since $\G_{c'(F)}$ is also strongly non-bipartite, $c'(F) \in \B(\G)$.  
By Lemma \ref{leaf}, we have $\mu^*(\G_{c(F)}) \ge \mu^*(\G_{c'(F)} ) \ge s(\G_{c(F)})$.
Therefore, $P_F \in P^t$ if $t \ge \mu^*(\G_{c'(F)})+1$ by Theorem \ref{stab}. 
\end{proof} 

\begin{Remark}
Chen, Morey and Sung  \cite[Corollary 4.3]{CMS} showed that if $\G$ is a connected non-bipartite graph, then
$$\astab(I) \le m - \ell,$$
where $m$ is the number of the vertices of degree $\ge 2$ and $2\ell+1$ is the smallest
length of an odd cycle in $\G$. This formula is an immediate consequence of Corollary \ref{astab 1}. 
In fact, if we let $U$ to be the set of vertices of degree $\ge 2$, then $U \in \B(\G)$.
Since $\G_U$ has generalized ear decompositions beginning with a cycle of length $2\ell+1$, we have $\varphi^*(\G_U) \le m - 2\ell - 1$. Hence, $\mu^*(\G_U) \le m - \ell$. 
As shown in \cite[Example 3.8]{HLT}, the bound $\astab(I) \le m - \ell$  is far from being optimal. 
\end{Remark}

The construction of $\Ass^\infty(I)$ by Chen, Morey and Sung in \cite{CMS} gives, 
for each $P_F \in \Ass^\infty(I)$, a number $s$ such that $P_F \in \Ass(I^t)$ for $t \ge s$. 
In the following example, we show that this number $s$ can be arbitrarily larger than $s(\G_{c(F)})+1$, 
the least number with this property. 

\begin{Example} 
Let $\G$ be the union of a triangle $C$ with the vertices $1,2,3$, an even cycle $D$ with the vertices $3,4,...,2r$,    and the edges $\{i,2r+i-3\}$, $i = 4,...,2r$, where $r \ge 3$ (see Figure 5 for the case $r = 3$).
Then $\G$ is a non-bipartite connected graph.
Since $c(V) = V$, we have $\G_{c(V)} = \G$. Hence, $\mm = P_V \in \Ass^\infty(I)$ by Corollary \ref{infty}. \par

\begin{figure}[ht!]
\begin{tikzpicture}[scale=0.6]

\draw [thick] (1,2) coordinate (a) -- (1,0) coordinate (b) ;
\draw [thick] (1,2) coordinate (a) -- (2.5,1) coordinate (c) ;
\draw [thick] (1,0) coordinate (b) -- (2.5,1) coordinate (c) ;
\draw [thick] (2.5,1) coordinate (c) -- (4,2) coordinate (d) ; 
\draw [thick] (2.5,1) coordinate (c) -- (4,0) coordinate (e);
\draw [thick] (4,0) coordinate (e) -- (5.5,1) coordinate (f);
\draw [thick] (4,2) coordinate (d) -- (5.5,1) coordinate (f);
\draw [thick] (5.5,1) coordinate (f) -- (7.5,1) coordinate (g);
\draw [thick] (4,2) coordinate (d) -- (4,4) coordinate (h) ;
\draw [thick] (4,0) coordinate (e) -- (4,-2) coordinate (k) ;
   
\fill (a) circle (3pt);
  \fill (b) circle (3pt);
  \fill (c) circle (3pt);
  \fill (d) circle (3pt);
  \fill (e) circle (3pt);
  \fill (f) circle (3pt);
  \fill (g) circle (3pt);
   \fill (h) circle (3pt);
    \fill (k) circle (3pt);
    
\draw (1.0,2) node[left, above] {$1$};
\draw (1,0) node[left, below] {$2$};
\draw (2.5,1) node[left, above] {$3$};
\draw (3.6,2) node[right,above] {$4$};
\draw (3.6,0) node[right,below] {$6$};
\draw (5.5,1) node[right,above] {$5$};
\draw (7.5,1) node[right,above] {$8$};
\draw (4.4,3.6) node[right,above] {$7$};
\draw (4.4,-1.6) node[right,below] {$9$};
\draw (1.5,1) node{$C$};
\draw (4,1) node{$D$};

\end{tikzpicture}
\caption{}
\end{figure} 

Let $s$ be the number given by the construction of Chen, Morey and Sung such that $\mm \in \Ass(I^t)$ for $t \ge s$. 
It is easy to check that $s = 4r-5$. We refer the reader to \cite{CMS} because their construction is too complicated to be recalled here. \par

Now, we are going to compute $s(\G)$, which is the minimum of all $\mu^*(\G_U)$, where
$U$ is a dominating set of $\G$ such that $\G_U$ is strongly non-bipartite.
Such an induced graph $\G_U$ must contain $C$ because $C$ is the unique odd cycle of $\G$.
Moreover, $U$ must contain the vertices $4,...,2r$. 
Indeed, if $U$ does not contain $i$ for some $i = 4,...,2r$, 
then $U$ contains $2r+i-3$, which becomes an isolated vertex of $\G_U$.
That cannot happen because $\G_U$ is strongly non-bipartite. 
Therefore, $\G_U$ contains the graph $C \cup D$.
If $2r+i -3\in U$ for some $i = 4,...,2r$, then $2r+i-3$ is a leaf vertex of $\G_U$.
From this it follows that $C \cup D$ is the graph obtained from $\G_U$ by removing all leaf vertices.
By Lemma \ref{leaf}, $\mu^*(C \cup D) \le \mu^*(\G_U)$. 
We may also put $U = \{1,...,2r\}$. In this case, $\G_U = C \cup D$. Therefore,
$$s(\G) = \mu^*(C \cup D).$$
The sequence $C,D$ form an odd-beginning generalized decomposition $\E$ of $C \cup D$ with $\varphi(\E) = 1$.
For every other odd-beginning generalized decomposition $\F$ of $\G$, we have $\varphi(\F) \ge 1$ because 
$\F$ contains at least a 2-cycle. Therefore, $\varphi^*(C \cup D) = 1$, which implies
$$\mu^*(C \cup D) = (\varphi^*(C \cup D) + 2r - 1)/2 = r.$$
So we get $s(\G) = r \le 4r - 5 = s$.
\end{Example}


\section{Minimal $s$-bases}

By Theorem \ref{associated}, to find the associated primes of $I^t$ 
we have to look for strongly non-bipartite reduced subgraphs $\G_U$ of $\G$ with $\mu^*(\G_U) < t$.
Therefore, it is of interest to classify all strongly non-bipartite graphs 
whose $\mu^*$-invariant is less than $t$ for every $t \ge 2$. 
The aim of this section is to show that we can recursively 
describe these graphs. 

\begin{Lemma} \label{persistent}
Let $\G'$ be a subgraph of $\G$ with the same vertex set. Then $\mu^*(\G) \le \mu^*(\G')$.
\end{Lemma}

\begin{proof}
The assumption implies that every odd-beginning generalized ear decomposition of $\G'$ is also an odd-beginning generalized ear decomposition of $\G$. From this it follows that $\varphi^*(\G) \le  \varphi^*(\G')$. Hence, $\mu^*(\G) \le  \mu^*(\G')$. 
\end{proof}

This lemma shows that we only need to classify the minimal strongly non-bipartite graphs 
whose $\mu^*$-invariant is equal a given number $s$, $s \ge 1$.  
For simplicity, we introduce the following notions.

\begin{Definition}
We call  a strongly non-bipartite graph $\G$ with $\mu^*(\G) = s$ an {\em $s$-base}. 
An $s$-base is {\em minimal} if it doesn't contain any other $s$-base with the same vertex set. 
\end{Definition}

We can describe the minimal associated primes of $I^t$ for any graph $\G$ if we know all minimal $s$-bases for $s < t$.

In the following,  the closed neighborhood of a subgraph means the closed neighborhood of its vertex set.

\begin{Theorem} \label{minimal base}
Let $F$ be a cover of $\G$. Then $P_F$ is an associated prime of $I^t$ if and only if $F$ is a minimal cover or $F$ is minimal among the covers containing the closed neighborhood of a minimal $s$-base in $\G$, $s < t$.
\end{Theorem}
 
\begin{proof} 
By Theorem \ref{associated}, it suffices to show that a reduced graph $\G_U$ is strongly non-bipartite with $\mu^*(\G_U) < t$ if and only if $\G_U$ contains a minimal $s$-base with the same vertex set $U$, $s < t$. By definition,
$\G_U$ contain a minimal $\mu^*(\G_U)$-base with the same vertex set $U$. Therefore, we only need to show 
that $\mu^*(\G_U) \le t$ if $\G_U$ contains a minimal $s$-base with the same vertex set $U$, $s < t$. But this follows from Lemma \ref{persistent}. 
\end{proof} 

\begin{Lemma} \label{1-base}
The triangle is the only 1-base.
\end{Lemma}

\begin{proof} 
By definition, $\mu^*(\G) = (\varphi^*(\G)+n-c)/2$, where $\varphi^*(\G)$ is the minimum number of even walks in odd-beginning ear decompositions of $\G$ and $c$ is the number of connected components of $\G$. 
Since every connected component of $\G$ contains an odd cycle, $n \ge 3c$.
Therefore, $\mu^*(\G) = 1$ if and only if $\varphi^*(\G) = 0$, $n = 3$, $c = 1$, which means $\G$ is a triangle. 
\end{proof}

Since the triangle is the only 1-base, we immediately obtain from Theorem \ref{minimal base}
the following results of Herzog and Hibi \cite{HH2} and Terai and Trung  \cite{TT2} on the associated primes of $I^2$.

\begin{Theorem} \cite[Theorem 3.8]{TT2}
Let $F$ be a  cover of $\G$. Then $P_F$ is an associated prime of $I^2$ if and only if $F$ is a minimal cover or $F$ is minimal among the covers containing the closed neighborhood of a triangle.
\end{Theorem}

\begin{Corollary}  \cite[Theorem 3.1]{HH2} \cite[Theorem 2.8]{TT2}
$\mm \in \Ass(I^2)$ if and only if $\G$ has a dominating triangle. 
\end{Corollary}

The following observation shows that we only need to describe connected minimal $s$-bases.

\begin{Lemma} \label{disconnect}
A graph $\G$ with connected components $\G_1,...,\G_c$ is a minimal $s$-base if and only if every component $\G_i$ is a minimal $s_i$-base for some number $s_i$ and $s = s_1+ \cdots + s_n$. 
\end{Lemma}

\begin{proof}
This follows from the additivity of the $\mu^*$-invariant on the connected components of a graph.
\end{proof}

Connected minimal bases can be described recursively as follows.  

\begin{Proposition} \label{recursive}
Let $\G$ be a connected minimal $s$-base, $s \ge 2$. 
Then $\G$ belongs to one of the following classes of graphs: \par
{\rm (i)} $\G$ is an odd cycle of length $2s+1$, \par
{\rm (ii)} $\G$ is a union of a connected minimal $(s-1)$-base $\D$ with an edge meeting $\D$ at only one vertex,\par 
{\rm (iii)} $\G$ is a union of a connected minimal $r$-base $\D$, $r < s$, with an ear of length $2(s-r)$ or $2(s-r)+1$ meeting $\D$ only at the endpoints.
\end{Proposition}

\begin{proof}   
Let $\E$ be an odd-beginning generalized ear decomposition of $\G$ with $\varphi(\E) = \varphi^*(\G)$. 
If $\E$ consists of only one walk, then this walk is an odd cycle of length $n$ and $\varphi(\E) = 0$.
Therefore, $\mu^*(\G) = (n -1)/2 = s$. Hence $n = 2s+1$. Since this odd cycle is a $s$-base, 
$\G$ must be this odd cycle. This is case (i). \par

Assume that $\E$ has more than one walk. Let $C$ be the last walk of $\E$.  
Let $\D$ be the induced graph on the vertices of the other walks of $\E$. 
Then $\D$ is non-bipartite because it contains the first odd cycle of $\E$.
Let $\F$ be the odd-beginning generalized ear decomposition of $\D$ obtained from $\E$ by removing $C$.
For every odd-beginning generalized ear decomposition $\F'$ of $\D$, we denote by $\E'$
 the odd-beginning generalized ear decomposition of $\G$ obtained by adding $C$ to $\F'$.
 Then $\varphi(\E) = \varphi^*(\G) \le \varphi(\E')$. This implies 
$\varphi(\F) \le \varphi(\F')$. Therefore, $\varphi(\F) = \varphi^*(\D)$. \par

Let $r = \mu^*(\D)$. 
If $\D$ is not a minimal $r$-base, $\D$ contains another $r$-base $\D'$ with the same vertex set.
Since $\mu^*(\D') = \mu^*(\D)$, we have 
$\varphi^*(\D') = \varphi^*(\D).$
Let $\G'$ denote the union of $\D'$ with $C$. 
Then $\G'$ is a strongly non-bipartite proper subgraph of $\G$ with the same vertex set.
Let $\F'$ be an odd-beginning generalized ear decomposition of $\D'$ with $\varphi(\F') = \varphi^*(\D')$.
Then $\F'$ is also an odd-beginning generalized ear decomposition of $\D$ with $\varphi(\F') = \varphi^*(\D)$.
From this it follows that $\varphi(\F') = \varphi(\F)$. 
Since $\E'$ and $\E$ are obtained from $\F'$ and $\F$ by adding the walk $C$, this implies
$\varphi(\E') = \varphi(\E) = \varphi^*(\G)$. Hence, $\varphi^*(\G') \le \varphi^*(\G)$.
On the other hand, since $\G'$ has the same vertex set as $\G$, every odd-beginning generalized ear decomposition of $\G'$ is also an odd-beginning generalized ear decomposition of $\G$. Hence, $\varphi^*(\G') \ge \varphi^*(\G)$. So we get $\varphi^*(\G') = \varphi^*(\G)$. This implies $\mu^*(\G') = \mu^*(\G) = s$, which contradicts the minimality of $\G$ as an $s$-base. Therefore, $\D$ is a minimal $r$-base.
\par 

Let $m$ be the number of vertices of $\D$. 
If $m =n$, $C$ must be an edge connecting two vertices of $\D$. Hence $\varphi(\F) = \varphi(\E)$. 
This implies $\varphi^*(\D) = \varphi^*(\G)$. 
Therefore, $\D$ is an $s$-base with the same vertex set $V$, contradicting the minimality of $\G$ as an $s$-base. 
Thus, $m < n$. Since $\varphi^*(\D) = \varphi(\F) \le \varphi(\E) = \varphi^*(\G)$, we have
$$r = \mu^*(\D) = (\varphi^*(\D)+m-1)/2 < (\varphi^*(\G)+n-1)/2 = \mu^*(\G) = s.$$

Let $\ell$ be the length of $C$. Then
\begin{align*}
\ell  = n - m +1 & = (2s - \varphi^*(\G) +1) - (2r - \varphi^*(\D) +1) +1\\
& = 2(s-r) - \varphi(\E)+\varphi(\F)+1.
\end{align*}
If $C$ is an even ear or a 2-cycle, $\varphi(\E) = \varphi(\F) +1$. Hence $\ell = 2(s-r)$, which implies $r = s-1$ if
$C$ is a 2-cycle. This is case (ii) and (iii) for an even ear.
If $C$ is an odd ear, $\varphi(\E) = \varphi(\F)$. Hence $\ell = 2(s-r)+1$. This is case (iii) for an odd ear. 
The proof of Proposition \ref{recursive} is now complete.
\end{proof}

Proposition \ref{recursive} only gives necessary conditions  for connected minimal $s$-bases.
The following proposition shows that in most cases, these necessary conditions are also sufficient.

\begin{Proposition} \label{s-base}
$\G$ is a minimal $s$-base if $\G$ belongs to one of the following cases: \par  
{\rm (i)} $\G$ is an odd cycle of length $2s+1$,\par
{\rm (ii)} $\G$ is a union of a connected minimal $(s-1)$-base $\D$ with an edge meeting $\D$ at only one vertex, \par
{\rm (iii)} $\G$ is a union of a connected minimal $r$-base $\D$ for some $r \le s-2$ with an even cycle of length $2(s-r)$ meeting $\D$ only at a vertex. \par
{\rm (iv)} $\G$ is a union of a connected minimal $r$-base $\D$ which is factor-critical for some $r < s$ with an odd cycle of length $2(s-r)+1$ meeting $\D$ only at a vertex. 
\end{Proposition} 

\begin{proof}
(i) We know that $\mu^*(\G) = s$. 
Since the cycle has no non-bipartite proper subgraph, $\G$ is a minimal $s$-base. \par

(ii) Let $E$ be the given edge.  The vertex of $E$ not in $\D$ is a leaf vertex of $\G$. 
Hence, $\mu^*(\G) = \mu^*(\D) +1 = s$ by Lemma \ref{leaf}. 
If $\G'$ is a strongly non-bipartite proper subgraph of $\G$ with the same vertex set, then $\G'$ contains $E$. 
Let  $\D'$ be the induced graph of $\G'$ on the vertex set of $\D$. 
Then $\D'$ is a strongly non-bipartite subgraph of $\D$ with the same vertex set. 
By the minimality of $\D$, we have $\mu^*(\D') \neq s-1$.
Applying Lemma \ref{leaf} to $\G'$,  we have
$\mu^*(\G') = \mu^*(\D')+1 \neq s.$
Therefore, $\G$ is a minimal $s$-base.\par

(iii) Let $\E$ be an arbitrary odd-beginning generalized ear decomposition of $\G$. 
Let $C$ be the given even cycle. 
Since $C$ meets $\D$ only at the endpoint, 
$\E$ can be decomposed as a union of an odd-beginning generalized ear decomposition $\F$ of $\D$ with a generalized ear decomposition $\mathcal G$ of $C$. It is clear that $\mathcal G$ is either $C$ or a sequence of more than 2-cycles.
If we assume furthermore that $\varphi(\E) = \varphi^*(\G)$, 
then $\F$ and $\mathcal G$ must have the smallest possible number of even walks. Therefore, $\varphi(\F) = \varphi^*(\D)$ and $\varphi(\mathcal G) = 1$ (i.e. $C$ is a walk of $\E$).
From this it follows that 
$$\varphi^*(\G) = \varphi(\E) = \varphi(\F)+\varphi(\mathcal G) = \varphi^*(\D)+1.$$
Let $m$ denote the number of the vertices of $\D$. Then 
$r = (\varphi^*(\D)+m-1)/2.$ Hence,
$n = m + 2(s-r) -1 = 2s - \varphi^*(\D).$
Therefore, 
$$\mu^*(\G) = (\varphi^*(\G)+ n-1)/2 = (\varphi^*(\D) + n)/ 2 = s.$$

Assume that $\G$ contains another $s$-base $\G'$ with the same vertex set. 
Then $\mu^*(\G') = \mu^*(\G)$. Hence, $\varphi^*(\G') = \varphi^*(\G)$.
By definition, every odd-beginning generalized ear decomposition $\E'$ of $\G'$ is also an odd-beginning generalized ear decomposition of $\G$. Choose $\E'$ such that $\varphi(\E') = \varphi^*(\G)$.
Similarly as above, we can show that $C$ is a walk of $\E'$.
Hence, $\G'$ contains $C$. 
Let $\D'$ denote the induced subgraph of $\G'$ on the vertex set of $\D$. 
Then $\G'$ is a union of $\D'$ and the even cycle $C$.
Since $\G'$ is a strongly non-bipartite proper subgraph of $\G$, 
$\D'$ is a strongly non-bipartite proper subgraph of $\D$ with the same vertex set.
Similarly as above, we can show that $\varphi^*(\G) = \varphi^*(\D')+1$. 
This implies $\varphi^*(\D') = \varphi^*(\D)$. Hence, $\mu^*(\D') = \mu^*(\D)$, 
which contradicts the minimality of $\D$ as an $r$-base. 
Therefore, $\G$ is a minimal $s$-base. \par

(iv) By Proposition \ref{Lovasz}, $\D$ has a generalized odd ear decomposition $\E$. 
From this it follows that $\varphi^*(\D) = 0$. 
Let $m$ be the number of the vertices of $\D$. 
Then $r = \mu^*(\D) = (m-1)/2$. 
Let $C$ be the given odd cycle of length $2(s-r)+1$. 
Since $\G$ is the union of $\D$ and $C$, which meet only at  a vertex,
$$n = m + 2(s-r) = 2s+1.$$
Adding $C$ to $\E$, we obtain an odd generalized ear decomposition of $\G$. 
From this it follows that $\varphi^*(\G) = 0$. 
Hence $\mu^*(\G) = (n -1)/2 = s$. \par

Assume that $\G$ contains another $s$-base $\G'$ with the same vertex set. 
Then $\varphi^*(\G') = \varphi^*(\G) = 0$. Hence, $\G'$ has an odd generalized ear decomposition $\E'$.
By definition, $\E'$ is also an odd generalized ear decomposition of $\G$.
As in (iii), we can show that $C$ is a walk of $\E'$. 
Let $\D'$ denote the induced subgraph of $\G'$ on the vertex set of $\D$. 
Then $\G'$ is the union of $\D'$ and $C$. Hence, $\D'$ is a proper subgraph of $\D$ with the same vertex set.
The walks of $\E'$ other than $C$ form an odd generalized ear decomposition of $\D'$.
Hence, $\varphi^*(\D') = 0 = \varphi^*(\D)$. This implies that $\mu^*(\D') = \mu^*(\D)$,
which contradicts the minimality of $\D$ as an $r$-base. Therefore, $\G$ is a minimal $s$-base.
\end{proof}

In view of Proposition \ref{recursive} (ii) and (iii), one may ask whether the addition of an ear of length  $2(s-r)$ or $2(s-r)+1$ for some  $s > r$ to a minimal $r$-base $\D$ always yields a minimal $s$-base. 
Proposition \ref{s-base}(iii) gives a positive answer to this question if the ear is an even cycle, whereas 
Proposition \ref{s-base}(iv) only gives a partial positive answer if the ear is an odd cycle.  
If the ear is a path, the answer is negative by the following example.

\begin{Example}   \label{not 2-base}
Let $\G_1$ and $\G_2$ be the union of a triangle $\D$ with a path of length 2 or 3 meeting $\D$ at the endpoints (see Figure 6). It is easy to check that  $\G_1$ and $\G_2$ are 2-bases. If we delete from $\G_1$ an edge of the path or from $\G_2$ the edge of $\D$ sharing the same endpoints with the path, we obtain subgraphs $\G'_1$ and $\G'_2$ with the same vertex sets, which are 2-bases. Therefore, $\G_1$ and $\G_2$ are not minimal 2-bases.

\begin{figure}[ht!]

\begin{tikzpicture}[scale=0.6]

\draw [thick] (0,1) coordinate (a) -- (1.5,2) coordinate (b) ;
\draw [thick] (0,1) coordinate (a) -- (1.5,0) coordinate (c) ;
\draw [thick] (1.5,2) coordinate (b) -- (1.5,0) coordinate (c) ; 
\draw [thick] (1.5,2) coordinate (b) -- (3,1) coordinate (d);
\draw [thick] (1.5,0) coordinate (c) -- (3,1) coordinate (d);

\fill (a) circle (3pt);
    \fill (b) circle (3pt);
  \fill (c) circle (3pt);
  \fill (d) circle (3pt);
  
\draw [thick] (6,1) coordinate (a') -- (7.5,2) coordinate (b') ;
\draw [thick] (6,1) coordinate (a') -- (7.5,0) coordinate (c') ;
\draw [thick] (7.5,2) coordinate (b') -- (7.5,0) coordinate (c') ;  
\draw [thick] (7.5,0) coordinate (c') -- (9,1) coordinate (d');

\fill (a') circle (3pt);
    \fill (b') circle (3pt);
  \fill (c') circle (3pt);
  \fill (d') circle (3pt);

\draw [thick] (13,0) coordinate (m) -- (11.5,1) coordinate (k) ; 
\draw [thick] (13,2) coordinate (l) -- (11.5,1) coordinate (k) ;
\draw [thick] (13,0) coordinate (m) -- (15,0) coordinate (p) ; 
\draw [thick] (13,2) coordinate (l) -- (15,2) coordinate (n) ;
\draw [thick] (15,0) coordinate (p) -- (15,2) coordinate (n) ;
\draw [thick] (13,0) coordinate (m) -- (13,2) coordinate (l) ;

\fill (k) circle (3pt);
\fill (l) circle (3pt);
\fill (m) circle (3pt);
\fill (n) circle (3pt);
\fill (p) circle (3pt);

\draw [thick] (19,0) coordinate (m') -- (17.5,1) coordinate (k') ; 
\draw [thick] (19,2) coordinate (l') -- (17.5,1) coordinate (k') ;
\draw [thick] (19,0) coordinate (m') -- (21,0) coordinate (p') ; 
\draw [thick] (19,2) coordinate (l') -- (21,2) coordinate (n') ;
\draw [thick] (21,0) coordinate (p') -- (21,2) coordinate (n') ;

\fill (k') circle (3pt);
\fill (l') circle (3pt);
\fill (m') circle (3pt);
\fill (n') circle (3pt);
\fill (p') circle (3pt);

\draw (1,1) node{$\D$};
\draw (1.8,-0.2) node[below] {$\G_1$};
\draw (7.8,-0.2) node[below] {$\G'_1$};
\draw (12.5,1) node{$\D$};
\draw (13.5,-0.2) node[below] {$\G_2$};
\draw (19.5,-0.2) node[below] {$\G'_2$};
\draw (4.5,1) node{$\longrightarrow$};
\draw (16.2,1) node{$\longrightarrow$};

\end{tikzpicture}
\caption{}
\end{figure}

\end{Example}

Using the above results we can easily find all minimal $2$-bases. 

\begin{Lemma} \label{2-base}
The minimal 2-bases are the four graphs in Figure 7.
\end{Lemma}

\begin{figure}[ht!]

\begin{tikzpicture}[scale=0.6]

\draw [thick] (0,0) coordinate (a) -- (0,2) coordinate (b) ;
\draw [thick] (0,2) coordinate (b) -- (1.5,1) coordinate (c) ;
\draw [thick] (1.5,1) coordinate (c) -- (0,0) coordinate (a) ; 
\draw [thick] (1.5,1) coordinate (c) -- (3.5,1) coordinate (d);

\fill (a) circle (3pt);
    \fill (b) circle (3pt);
  \fill (c) circle (3pt);
  \fill (d) circle (3pt);

\draw [thick] (6,0) coordinate (e) -- (6,2) coordinate (f) ;
\draw [thick] (6,2) coordinate (f) -- (7.5,1) coordinate (g) ;
\draw [thick] (7.5,1) coordinate (g) -- (6,0) coordinate (e) ; 
\draw [thick] (7.5,1) coordinate (g) -- (9,0) coordinate (h) ; 
\draw [thick] (7.5,1) coordinate (g) -- (9,2) coordinate (i) ; 
\draw [thick] (9,0) coordinate (h) -- (9,2) coordinate (i) ;

\fill (e) circle (3pt);
\fill (f) circle (3pt);
\fill (g) circle (3pt);
\fill (h) circle (3pt);
\fill (i) circle (3pt);

\draw [thick] (13,0) coordinate (m) -- (11.5,1) coordinate (k) ; 
\draw [thick] (13,2) coordinate (l) -- (11.5,1) coordinate (k) ;
\draw [thick] (13,0) coordinate (m) -- (15,0) coordinate (p) ; 
\draw [thick] (13,2) coordinate (l) -- (15,2) coordinate (n) ;
\draw [thick] (15,0) coordinate (p) -- (15,2) coordinate (n) ;

\fill (k) circle (3pt);
\fill (l) circle (3pt);
\fill (m) circle (3pt);
\fill (n) circle (3pt);
\fill (p) circle (3pt);

\draw [thick] (17.5,0) coordinate (x) -- (17.5,2) coordinate (y) ;
\draw [thick] (17.5,2) coordinate (y) -- (19,1) coordinate (z) ;
\draw [thick] (19,1) coordinate (z) -- (17.5,0) coordinate (x) ; 
\draw [thick] (22,2) coordinate (u) -- (20.5,1) coordinate (t) ; 
\draw [thick] (22,0) coordinate (v) -- (20.5,1) coordinate (t) ;
\draw [thick] (22,0) coordinate (v) -- (22,2) coordinate (u) ;

\fill (x) circle (3pt);
\fill (y) circle (3pt);
\fill (z) circle (3pt);
\fill (t) circle (3pt);
\fill (u) circle (3pt);
\fill (v) circle (3pt);
\draw (20.4,1) node[left] {+};
\end{tikzpicture}
\caption{}
\end{figure}

\begin{proof}
By Proposition \ref{recursive}, a connected minimal $2$-base can be only one of the following graphs: a 
 pentagon,  a union of a triangle with an edge meeting the triangle at only a vertex, or a union of the triangle with an ear of length $2$ or $3$ meeting the triangle only at the endpoints (or endpoint if the ear is a cycle). 
The first two graphs and the union of two triangles meeting at only one vertex 
are minimal 2-bases by Proposition \ref{s-base} (i), (ii), and (iv), respectively.
The remained graphs are the unions of a triangle with a path of length 2 or 3, which are not minimal 2-bases by Example \ref{not 2-base}. By Lemma \ref{disconnect}, a disconnected minimal 2-base must be the union of two disjoint triangles because the triangle is the only 1-base. Thus, a minimal 2-base must be one of the four graphs in Figure 7. 
\end{proof}

Combining Theorem \ref{minimal base} with Lemma \ref{1-base} and Lemma \ref{2-base} we immediately obtain the following results of Hien, Lam and Trung \cite{HLT} on the associated primes of $I^3$.

\begin{Theorem} \label{asso 3}  \cite[Theorem 4.4]{HLT}
Let $F$ be a cover of $\G$. Then $P_F$ is an associated prime of $I^3$ if and only if $F$ is a minimal cover or $F$ is minimal among the covers containing the closed neighborhood of a triangle or of one of the four graphs in Figure 7.
\end{Theorem}


\begin{thebibliography}{1}

\bibitem{Ba} A. Banerjee, {\it The regularity of powers of edge ideals}, J. Algebraic Combin. {\bf 41} (2015), no. 2, 303--321. 
\bibitem{Br} M. Brodmann, {\it Asymptotic stability of $\Ass(M/I^nM)$}, Proc. Amer. Math. Soc. 74 (1979), 16--18.
\bibitem{Br2} M. Brodmann, {\it The asymptotic nature of the analytic spread}, Math. Proc. Cambridge Philos. Soc. {\bf 86} (1979), 35--39.
\bibitem{BH} W. Bruns and J. Herzog, Cohen-Macaulay rings, Cambridge Stud. in Adv. Math. {\bf 39}, Cambridge University Press, Cambridge, 1993.
\bibitem{CMS} J. Chen, S. Morey and A. Sung, 
{\it The stable set of associated primes of the ideal of a graph}, Rocky Mountain J. Math. {\bf 32} (2002), 71--89.
\bibitem{CHT}
S.~Cutkosky, J.~Herzog, and N.~V. Trung,
{\it Asymptotic behaviour of the Castelnuovo-Mumford regularity},
Compos. Math. {\bf 118} (1999), 243--261.
\bibitem{DTTY} N.T. Dung, N.T.T. Tam, H.L. Truong, and H.N. Yen, {\em Critical paired dominating sets and irreducible decompositions of powers of edge ideals}, Acta Math. Vietnam. (2018). https://doi.org/10.1007/s40306-018-0273-0
\bibitem{E} D. Eisenbud, Commutative algebra. With a view toward algebraic geometry. Graduate Texts in Math. {\bf 150}, Springer-Verlag, New York, 1995.
\bibitem{EH}
D. Eisenbud and J. Harris, 
{\it Powers of ideals and fibers of morphisms},
Math. Res. Lett. {\bf 17} (2010), no. 2, 267--273.
\bibitem{EU}
D.~Eisenbud and B. Ulrich,
{\it Notes on regularity stabilization}, 
Proc. Amer. Math. Soc. {\bf 140} (2012), no. 4, 1221--1232.
\bibitem{FHV} C.A. Francisco, H.T. Ha and A. Van Tuyl, 
{\it Colorings of hypergraphs, perfect graphs, and associated primes of powers of monomial ideals}, 
J. Algebra {\bf 331} (2011), 224--242.
\bibitem{Fr} A. Frank, 
{\it Conservative weightings and ear-decompositions of graphs}, 
Combinatorica {\bf 13} (1993), 65--81.
\bibitem{Ga} T. Gallai, {\it Neuer Beweis eines Tutte'schen Satzes}, Magyar Tud. Akad. Mat. Kutato. Int. K\"ozl. {\bf 8} (1963), 135--139.  
\bibitem{HM} H.T. Ha and S. Morey, {\it Embedded associated primes of powers of squarefree monomial ideals}, J. Pure Appl. Algebra {\bf 214} (2010), 301--308.
\bibitem{HS} H.T. Ha and M. Sun, {\it Squarefree monomial ideals that fail the persistence property and non-increasing depth}, Acta Math. Vietnam. {\bf 40} (2015), 125--137.
\bibitem{HH} J. Herzog and T. Hibi, 
{\it The depth of powers of an ideal}, 
J. Algebra {\bf 291} (2005), no. 2, 534--550.
\bibitem{HH1} J. Herzog and T. Hibi, Monomial ideals, Grad. Texts in Math. {\bf 260}, Springer-Verlag London, Ltd., London, 2011.
\bibitem{HH2} J. Herzog and T. Hibi, 
{\it Bounding the socles of powers of squarefree monomial ideals}, 
in:  D. Eisenbud et al (editors), Commutative Algebra and Noncommutative Algebraic Geometry, Vol. 2, 
 MSRI Publications {\bf 68} (2015), 223--229.
\bibitem{HHT} J. Herzog, T. Hibi and N.V. Trung, {\it Symbolic powers of monomial ideals and vertex cover algebras}, Adv. Math. {\bf 210} (2007), 304--322.
\bibitem{HQ} J. Herzog and A. Qureshi, {\it Persistence and stability properties of powers of ideals}, J. Pure Appl. Algebra {\bf 219} (2015), no. 3, 530--542.
\bibitem{HL} H.T. Hien and H.M. Lam,  
{\it Combinatorial characterizations of the saturation and the associated primes
 of the fourth power of edge ideals},  Acta Math. Vietnam. {\bf 40} (2015), 511--526.
\bibitem{HLT} H.T. Hien, H.M. Lam and N.V. Trung, {\it Saturation and associated primes of powers of edge ideals}, J. Algebra {\bf 439} (2015), 225--249.
\bibitem{Ho} L. T. Hoa,
{\it Stability of associated primes of monomial ideals}, 
Vietnam J. Math. {\bf 34} (2006), 473--487.
\bibitem{Ko} V.~Kodiyalam,
{\it Asymptotic behaviour of Castelnuovo-Mumford regularity},
Proc. Amer. Math. Soc. {\bf 128} (2000), 407--411.
\bibitem{Lo} L. Lovasz, {\it A note on factor-critical graphs}, Studia Sci. Math. Hungar. {\bf 7} (1972), 279--280.
\bibitem{MMV} J. Martinez-Bernal, S. Morey and R.H. Villarreal, {\it Associated primes of powers of edge ideals},
 Collect. Math. {\bf 63} (2012), 361--374.
  \bibitem{Mo} S. Morey, {\it Depths of powers of the edge ideal of a tree}, Comm. Algebra {\bf 38} (2010), 
4042-4055.
\bibitem{MV} S. Morey and R.H. Villarreal,  {\it Edge ideals: algebraic and combinatorial properties}, in: Progress in Commutative Algebra 1, 85--126, de Gruyter, Berlin, 2012.
\bibitem{RTY} G. Rinaldo, N. Terai, and K. Yoshida, {\it Cohen--Macaulayness for symbolic power ideals of edge ideals}, J. Algebra {\bf 347} (2011), 405--430. 
\bibitem{Ro} H. E. Robbins, 
{\it A theorem on graphs, with an application to a problem of traffic control}, 
Amer. Math. Monthly {\bf 46} (1939), 281--283.
\bibitem{Sh} R.Y. Sharp, {\it Convergence of sequences of sets of associated primes},  Proc. Amer. Math. Soc, {\bf 131} (2003), 3009--3017. 
\bibitem{SVV} A. Simis, W. Vasconcelos and R. Villarreal, 
{\em On the ideal theory of graphs},  
J. Algebra  {\bf 167}  (1994),  no. 2, 389--416.
\bibitem{SZ} P. Sole and T. H. Zaslavsky, 
{\it The covering radius of the cycle code of a graph}, 
Discrete Math. {\bf 45} (1993), 63--70.
\bibitem{Ta} Y. Takayama, {\it Combinatorial characterizations of generalized Cohen-Macaulay monomial ideals}, Bull. Math. Soc. Sci. Math. Roumanie (N.S.) {\bf 48} (2005), 327--344.
\bibitem{TT2} N. Terai and N.V. Trung, {\it On the associated primes and the depth of the second power of squarefree monomial ideals},  J. Pure Appl. Algebra {\bf 218} (2014), 1117--1129.
\bibitem{TNT} T.N. Trung, {\it Stability of depth of power of edge ideals}, J. Algebra {\bf 452} (2016), 157--187.
\bibitem{We} D.B. West, Introduction to Graph Theory, 2nd ed., Prentice-Hall, 2001. 
\end{thebibliography}
\end{document}